\documentclass[reqno,12pt,a4paper]{amsart}

\usepackage{tikz}
\usetikzlibrary{matrix,arrows,calc,snakes,patterns,decorations.markings}

\usepackage[mathscr]{eucal}
\usepackage{graphics,epic,url}
\usepackage{amsfonts}
\usepackage{amscd}
\usepackage{latexsym}
\usepackage{amsmath,amssymb, amsthm, stmaryrd, bm, bbm}
\usepackage[all,2cell]{xy}
\usepackage{mathrsfs}

\newtheorem{theorem}{Theorem}[section]
\newtheorem*{theorem*}{Theorem}
\newtheorem*{theorem2.7*}{Theorem 2.7}

\newtheorem{lemma}[theorem]{Lemma}
\newtheorem*{lemma2.6*}{Lemma 2.6}
\newtheorem{proposition}[theorem]{Proposition}
\newtheorem{corollary}[theorem]{Corollary}

\newtheorem*{conjecture*}{Conjecture}
\newtheorem{question}[theorem]{Question}
\newtheorem*{question*}{Question}
\theoremstyle{remark}
\newtheorem{remark}[theorem]{Remark}
\newtheorem{example}[theorem]{Example}
\theoremstyle{definition}
\newtheorem{definition}[theorem]{Definition}
\newtheorem{setup}[theorem]{Setup}

\newcommand{\prdim}{\opname{pr.dim}\nolimits}
\newcommand{\Aus}{\opname{Aus}\nolimits}
\newcommand{\MCM}{\opname{MCM}\nolimits}
\newcommand{\krdim}{\opname{kr.dim}\nolimits}

\newcommand{\confer}{{\em cf.~}\ }
\newcommand{\eg}{{\em e.g.~}\ }

\newcommand{\ul}[1]{\underline{#1}}
\newcommand{\ol}[1]{\overline{#1}}

\newcommand{\opname}[1]{\operatorname{\mathsf{#1}}}

\renewcommand{\mod}{\opname{mod}\nolimits}
\newcommand{\fdmod}{\opname{fdmod}\nolimits}

\newcommand{\proj}{\opname{proj}\nolimits}

\newcommand{\Mod}{\opname{Mod}\nolimits}

\newcommand{\add}{\opname{add}\nolimits}

\newcommand{\im}{\opname{im}\nolimits}
\renewcommand{\ker}{\opname{ker}\nolimits}

\newcommand{\thick}{\opname{thick}\nolimits}
\newcommand{\Tria}{\opname{Tria}\nolimits}
\newcommand{\per}{\opname{per}\nolimits}
\newcommand{\Perf}{\opname{Perf}\nolimits}
\newcommand{\Spec}{\opname{Spec}\nolimits}

\DeclareMathOperator{\Coh}{\mathsf{Coh}}

\newcommand{\Z}{\mathbb{Z}}

\newcommand{\C}{\mathbb{C}}

\newcommand{\F}{\mathbb{F}}
\newcommand{\G}{\mathbb{G}}
\newcommand{\I}{\mathbb{I}}
\newcommand{\ra}{\rightarrow}

\newcommand{\id}{\mathbf{1}}

\newcommand{\Hom}{\opname{Hom}}
\newcommand{\End}{\opname{End}}

\newcommand{\RHom}{\opname{RHom}}
\newcommand{\cHom}{\mathcal{H}\it{om}}
\newcommand{\cEnd}{\mathcal{E}\it{nd}}

\newcommand{\Ext}{\opname{Ext}}

\newcommand{\ten}{\otimes}
\newcommand{\lten}{\overset{\opname{L}}{\ten}}

\newcommand{\gldim}{\opname{gldim}\nolimits}

\newcommand{\ca}{{\mathcal A}}

\newcommand{\cc}{{\mathcal C}}
\newcommand{\cd}{{\mathcal D}}
\newcommand{\ce}{{\mathcal E}}
\newcommand{\cf}{{\mathcal F}}

\newcommand{\ch}{{\mathcal H}}

\newcommand{\ck}{{\mathcal K}}

\newcommand{\co}{{\mathcal O}}
\newcommand{\cp}{{\mathcal P}}

\newcommand{\cs}{{\mathcal S}}
\newcommand{\ct}{{\mathcal T}}
\newcommand{\cu}{{\mathcal U}}
\newcommand{\cv}{{\mathcal V}}

\renewcommand{\hat}[1]{\widehat{#1}}

\numberwithin{equation}{section}

\begin{document}

\title[Relative singularity categories I]{Relative singularity categories I: \\[1ex] Auslander resolutions}

\dedicatory{Dedicated to Idun Reiten on the occasion of her 70th birthday}
\author{Martin Kalck}
\thanks{M.K. was supported by DFG grant Bu--1866/2--1 and EPSRC grant EP/L017962/1.}
\address{Martin Kalck, The Maxwell Institute, School of Mathematics, James Clerk Maxwell Building, The King's Buildings, Mayfield Road, Edinburgh, EH9 3JZ, UK.}
\email{m.kalck@ed.ac.uk}

\author{Dong Yang}
\subjclass[2010]{Primary 14B05, 14E15; Secondary 18E30}
\thanks{D.Y. was supported by the DFG program SPP 1388 (YA297/1-1 and KO1281/9-1), a JSPS postdoctoral fellowship program (P12318) and the National Science Foundation in China No. 11401297.}

\address{Dong Yang, Department of Mathematics, Nanjing University, Nanjing 210093, PR China}
\email{yangdong@nju.edu.cn}

\begin{abstract}
Let $R$ be an isolated Gorenstein singularity with a non-commutative resolution $A=\End_R(R\oplus M)$. In this paper, we show that the relative singularity category $\Delta_R(A)$ of $A$ has a number of pleasant properties, such as being Hom-finite. Moreover, it determines the classical singularity category $\cd_{sg}(R)$ of Buchweitz and Orlov as a certain canonical quotient category.  If $R$ has finite CM type, which includes for example Kleinian singularities, then we show the much more surprising result that $\cd_{sg}(R)$ determines $\Delta_R(\Aus(R))$, where $\Aus(R)$ is the corresponding Auslander algebra.  The proofs of these results use dg algebras, $A_\infty$ Koszul duality, and the new concept of dg Auslander algebras, which may be of independent interest.

\noindent{\bf Keywords}: isolated Gorenstein singularity, non-commutative resolution, singularity category, relative singularity category, dg Auslander algebra.

%
\end{abstract}

\maketitle
\tableofcontents

\section{Introduction}
Triangulated categories of singularities were introduced and studied by Buchweitz \cite{Buchweitz87} and later also by Orlov \cite{Orlov04, Orlov09, Orlov11} who related them to Kontsevich's Homological Mirror Symmetry Conjecture. They may be seen as a categorical measure for the complexity of the singularities of a Noetherian scheme $X$. If $X$ has only isolated Gorenstein singularities $x_{1}, \ldots, x_{n}$, then the singularity category is triangle equivalent to the direct sum of the stable categories of maximal Cohen--Macaulay $\widehat{\co}_{x_{i}}$-modules (up to direct summands) \cite{Buchweitz87, Orlov11}.

Starting with Van den Bergh's works \cite{VandenBergh04, NCCR}, non-commutative analogues of (crepant) resolutions (NC(C)R) of singularities have been studied intensively in recent years. Non-commutative resolutions are useful even if the primary interest lies in commutative questions: for example, the Bondal-Orlov Conjecture concerning derived equivalences between (commutative) crepant resolutions and the derived McKay-Correspondence \cite{BridgelandKingReid01, KapranovVasserot00} led Van den Bergh to the notion of an NCCR. Moreover, moduli spaces of quiver representations provide a very useful technique to obtain commutative resolutions from non-commutative resolutions, see e.g.~\cite{NCCR, WemyssReconstructionTypeA}.

Inspired by the construction of the singularity category, Burban and the first author introduced and studied the notion of \emph{relative singularity categories} \cite{BurbanKalck11}. These categories measure the difference between the derived category of a non-commutative resolution (NCR) \cite{DaoIyamaTakahashiVial12} and the smooth part $K^b(\proj-R) \subseteq \cd^b(\mod-R)$ of the derived category of the singularity. Continuing this line of investigations, this article focuses on the relation between relative and classical singularity categories. 

The techniques developed in this article led to a `purely commutative' result: in joint work with Iyama and Wemyss \cite{KIWY12}, we decompose Iyama \& Wemyss' `new triangulated category' for complete rational surface singularities  \cite{IyamaWemyss09} into blocks of singularity categories of ADE-singularities.
Moreover, using relative singularity categories, Van den Bergh \& Thanhoffer de V\"olcsey showed \cite{ThanhofferdeVolcseyMichelVandenBergh10} that the stable category of a complete Gorenstein quotient singularity of Krull dimension three is a generalized cluster category \cite{Amiot09, Guolingyan11a}. We recover some of their results using quite different techniques. We proceed with a more detailed outline of the results of this article.
\subsection{Setup}\label{ss:SetupIntro}
Let $k$ be an algebraically closed field. Let $(R, \mathfrak{m})$ be a commutative local complete Gorenstein $k$-algebra such that $k\cong R/\mathfrak{m}$. Let 
\begin{align}\label{mcm}
\MCM(R)=\left\{\left.M \in \mod-R \, \right| \Ext^i_{R}(M, R)=0 \text{ for all } i>0 \right\}
\end{align}
be the full subcategory of \emph{maximal Cohen--Macaulay} $R$-modules. 
Let $M_{0}=R$, $M_{1}$, $\ldots$, $M_{t}$ be pairwise non-isomorphic indecomposable $\MCM$ $R$-modules and $A=\End_{R}(M:=\bigoplus_{i=0}^t M_{i})$. If $\gldim(A)< \infty$ then $A$ is called a \emph{non-commutative resolution} (NCR) of $R$ (\confer \cite{DaoIyamaTakahashiVial12}). For example, if $R$ has only finitely many indecomposable $\MCM$s and $M$ denotes their direct sum, then the \emph{Auslander algebra} $\Aus(\MCM(R)):=\End_{R}(M)$ is an NCR \cite[Theorem A.1]{Auslander84}.  There is a fully faithful triangle functor
$
K^b(\proj-R) \rightarrow \cd^b(\mod-A),
$
whose essential image equals $\thick(eA) \subseteq \cd^b(\mod-A)$, where $e \in A$ is the idempotent corresponding to the projection on $R$.

\begin{definition}
The \emph{relative singularity category} is the Verdier quotient category \begin{align}\label{E:DefRelSingCat} \Delta_{R}(A):=\frac{\cd^b(\mod-A)}{K^b(\proj-R)} \cong \frac{\cd^b(\mod-A)}{\thick(eA)}.\end{align} 
\end{definition}

\begin{definition}
The \emph{classical singularity category} is the Verdier quotient category 
\begin{align} \cd_{sg}(R):=\cd^b(\mod-R)/K^b(\proj-R). \end{align} 
\end{definition}

\begin{theorem}[Buchweitz \cite{Buchweitz87}]
$\ul{\MCM}(R) \cong \cd_{sg}(R)$ as triangulated categories.
\end{theorem}In the sequel, we often use this result to identify these two categories. 


\subsection{Main Result} It is natural to ask how the two concepts of singularity categories defined above are related. Our main result gives a first answer to this question.
\theoremstyle{theorem}
\newtheorem*{mainthm}{Theorem \ref{t:Classical-versus-generalized-singularity-categories}}
\begin{mainthm}
 Let $R$ and $R'$ be $\MCM$--representation finite complete Gorenstein $k$-algebras with Auslander algebras $A=\Aus(\MCM(R))$ and $A'=\Aus(\MCM(R'))$, respectively. Then the following statements are equivalent.
\begin{itemize}
\item[(i)] There is an equivalence $\ul{\MCM}(R) \cong \ul{\MCM}(R')$ of triangulated categories.
\item[(ii)] There is an equivalence $\Delta_{R}\!\left(A\right) \cong \Delta_{R'}\!\left(A'\right)$ of triangulated categories.
\end{itemize}
The implication $(ii) \Rightarrow (i)$ holds more generally for non-commutative resolutions $A$ and $A'$ of arbitrary isolated Gorenstein singularities $R$ and $R'$, respectively.
\end{mainthm} 
\noindent Kn\"orrer's periodicity theorem \cite{Knoerrer87, Solberg89} yields a wealth of non-trivial examples for $(i)$: 
\begin{align}\label{E:Knoerrer}
\ul{\MCM}\big(S/(f)\big) \stackrel{\sim}\longrightarrow \ul{\MCM}\big(S\llbracket x, y \rrbracket/(f + xy )\big), 
\end{align}
where $\mathrm{char} \, k \neq 2$, $S=k\llbracket z_{0}, \ldots, z_{d}\rrbracket$,  $f$ is a non-zero element in $(z_{0}, \ldots, z_{d})$ and $d \geq 0$.

\begin{example} Let $R=\C\llbracket x \rrbracket/(x^2)$ and $R'=\C\llbracket x, y, z \rrbracket/(x^2+yz)$. Kn\"orrer's equivalence (\ref{E:Knoerrer}) in conjunction with Theorem \ref{t:Classical-versus-generalized-singularity-categories} above, yields a triangle equivalence $\Delta_{R}(\Aus(\MCM(R))) \cong \Delta_{R'}(\Aus(\MCM(R')))$, which may be written explicitly as  
\begin{align}
\frac{\cd^b\left(\left.
\begin{xy}
\SelectTips{cm}{10}
\xymatrix{1 \ar@/^10pt/[rr]|{\,\, pr\,\,}  && 2 \ar@/^10pt/[ll]|{\,\, incl \,\,} }\end{xy}\right/ (pr \circ incl)
\right)}{K^b(\add P_{1})}
\stackrel{\sim}\longrightarrow
\frac{\cd^b\left(\left.
\begin{xy}
\SelectTips{cm}{10}
\xymatrix{1 \ar@<-1.5pt>@/^10pt/[rr]|{\,\, y\,\,} \ar@<5pt>@/^10pt/[rr]|{\,\, x\,\,} && 2 \ar@<-1.5pt>@/^10pt/[ll]|{\,\, y\,\,} \ar@<5pt>@/^10pt/[ll]|{\,\, x\,\,}}
\end{xy}\right/ (xy-yx)
\right)}{K^b(\add P_{1})}.
\end{align}
\noindent The quiver algebra on the right is the completion of the preprojective algebra of the Kronecker quiver $\begin{xy}\SelectTips{cm}{10}\xymatrix{\circ \ar@/^/[r] \ar@/_/[r] &  \circ}\end{xy}$. The derived McKay--Correspondence \cite{KapranovVasserot00, BridgelandKingReid01} shows that these categories are triangle equivalent to the quotient category
\begin{align}
\Delta_{R'}(Y):=\frac{\cd^b(\Coh Y)}{\pi^*\left(\Perf(R')\right)},
\end{align}
where $\pi \colon Y \to \Spec(R')$ denotes the minimal resolution of singularities.
\end{example}

\subsection{Idea of the proof}
We prove Theorem \ref{t:Classical-versus-generalized-singularity-categories} by developing a 
general dg algebra framework. 
More precisely, to every Hom-finite idempotent complete algebraic triangulated category $\ct$ with finitely many indecomposable objects
satisfying a certain extra condition
(e.g.~this holds for $\ct=\ul{\MCM}(R)$), we associate  a 
dg algebra $\Lambda_{dg}(\ct)$ called the \emph{dg Auslander algebra of $\ct$} (Definition~\ref{d:dg-aus-alg}). It is completely determined by the triangulated category $\ct$. Now, using recollements generated by idempotents, Koszul duality and the fractional Calabi--Yau property \eqref{E:IntroFractCY}, we prove the existence of an equivalence of triangulated categories (Theorem \ref{t:main-thm-2})
\begin{align}
\Delta_{R}\Bigl(\Aus\bigl(\MCM(R)\bigr)\Bigr) \cong \per\Bigl(\Lambda_{dg}\bigl(\ul{\MCM}(R)\bigr)\Bigr). 
\end{align}
In particular, this shows that (i) implies (ii). Conversely, written in this language, the quotient functor (\ref{E:class-as-quotient}), induces an equivalence of triangulated categories
\begin{align}\label{E:Cluster-Equiv}
\frac{\per\Big(\Lambda_{dg}\big(\ul{\MCM}(R)\big)\Big)}{\cd_{fd}\Big(\Lambda_{dg}\big(\ul{\MCM}(R)\big)\Big)} \longrightarrow \ul{\MCM}(R).
\end{align} Since the category $\cd_{fd}(\Lambda_{dg}(\ul{\MCM}(R)))$ of dg modules with finite-dimensional total cohomology admits an intrinsic characterization in $\per(\Lambda_{dg}(\ul{\MCM}(R)))$, this proves that $\ul{\MCM}(R)$ is determined by $\Delta_{R}(\Aus(\MCM(R)))$. Hence, (ii) implies (i).

\newtheorem*{example*}{Example}
\begin{example*}
Let $R=\C\llbracket z_{0}, \ldots, z_{d}\rrbracket/(z_{0}^{n+1} + z_{1}^2 + \ldots + z_{d}^2)$ be an $A_{n}$-singularity of \emph{even} Krull dimension. Then the graded quiver $Q$ of $\Lambda_{dg}(\ul{\MCM}(R))$ is given as

\begin{equation*}\begin{tikzpicture}[description/.style={fill=white,inner sep=2pt}]
    \matrix (n) [matrix of math nodes, row sep=3em,
                 column sep=2.5em, text height=1.5ex, text depth=0.25ex,
                 inner sep=0pt, nodes={inner xsep=0.3333em, inner
ysep=0.3333em}]
    {  
       1 & 2 & 3 &\cdots& n-1 &n \\
    };
    
    \draw[->] ($(n-1-1.east) + (0mm,1mm)$) .. controls +(2.5mm,1mm) and
+(-2.5mm,+1mm) .. ($(n-1-2.west) + (0mm,1mm)$);
  \node[scale=0.75] at ($(n-1-1.east) + (5.5mm, 3.8mm)$) {$\alpha_{1}$};
  
    \draw[->] ($(n-1-2.west) + (0mm,-1mm)$) .. controls +(-2.5mm,-1mm)
and +(+2.5mm,-1mm) .. ($(n-1-1.east) + (0mm,-1mm)$);
\node[scale=0.75] at ($(n-1-1.east) + (5.5mm, -4.1mm)$) {$\alpha^*_{1}$};

    \draw[->] ($(n-1-2.east) + (0mm,1mm)$) .. controls +(2.5mm,1mm) and
+(-2.5mm,+1mm) .. ($(n-1-3.west) + (0mm,1mm)$);
  \node[scale=0.75] at ($(n-1-2.east) + (5.5mm, 3.8mm)$) {$\alpha_{2}$};
  
    \draw[->] ($(n-1-3.west) + (0mm,-1mm)$) .. controls +(-2.5mm,-1mm)
and +(+2.5mm,-1mm) .. ($(n-1-2.east) + (0mm,-1mm)$);
  \node[scale=0.75] at ($(n-1-2.east) + (5.5mm, -4.1mm)$) {$\alpha^*_{2}$};

    \draw[->] ($(n-1-3.east) + (0mm,1mm)$) .. controls +(2.5mm,1mm) and
+(-2.5mm,+1mm) .. ($(n-1-4.west) + (0mm,1mm)$);
  \node[scale=0.75] at ($(n-1-3.east) + (5.5mm, 3.8mm)$) {$\alpha_{3}$};

    \draw[->] ($(n-1-4.west) + (0mm,-1mm)$) .. controls +(-2.5mm,-1mm)
and +(+2.5mm,-1mm) .. ($(n-1-3.east) + (0mm,-1mm)$);
  \node[scale=0.75] at ($(n-1-3.east) + (5.5mm, -4.1mm)$) {$\alpha^*_{3}$};
   
    \draw[->] ($(n-1-4.east) + (0mm,1mm)$) .. controls +(2.5mm,1mm) and
+(-2.5mm,+1mm) .. ($(n-1-5.west) + (0mm,1mm)$);
  \node[scale=0.75] at ($(n-1-4.east) + (5.5mm, 3.8mm)$) {$\alpha_{n-2}$};

    \draw[->] ($(n-1-5.west) + (0mm,-1mm)$) .. controls +(-2.5mm,-1mm)
and +(+2.5mm,-1mm) .. ($(n-1-4.east) + (0mm,-1mm)$);
  \node[scale=0.75] at ($(n-1-4.east) + (5.5mm, -4.1mm)$) {$\alpha^*_{n-2}$};

    \draw[->] ($(n-1-5.east) + (0mm,1mm)$) .. controls +(2.5mm,1mm) and
+(-2.5mm,+1mm) .. ($(n-1-6.west) + (0mm,1mm)$);
  \node[scale=0.75] at ($(n-1-5.east) + (5.5mm, 3.8mm)$) {$\alpha_{n-1}$};

    \draw[->] ($(n-1-6.west) + (0mm,-1mm)$) .. controls +(-2.5mm,-1mm)
and +(+2.5mm,-1mm) .. ($(n-1-5.east) + (0mm,-1mm)$);
  \node[scale=0.75] at ($(n-1-5.east) + (5.5mm, -4.1mm)$) {$\alpha^*_{n-1}$};

 \draw[dash pattern = on 0.5mm off 0.3mm,->] ($(n-1-1.south) +
    (1.2mm,0.5mm)$) arc (65:-245:2.5mm);
    \node[scale=0.75] at ($(n-1-1.south) + (0mm,-6.5mm)$) {$\rho_{1}$};
 
  \draw[dash pattern = on 0.5mm off 0.3mm,->] ($(n-1-2.south) +
    (1.2mm,0.5mm)$) arc (65:-245:2.5mm);
   \node[scale=0.75] at ($(n-1-2.south) + (0mm,-6.5mm)$) {$\rho_{2}$};
 
   \draw[dash pattern = on 0.5mm off 0.3mm,->] ($(n-1-3.south) +
    (1.2mm,0.5mm)$) arc (65:-245:2.5mm);
   \node[scale=0.75] at ($(n-1-3.south) + (0mm,-6.5mm)$) {$\rho_{3}$};
   
     \draw[dash pattern = on 0.5mm off 0.3mm,->] ($(n-1-5.south) +
    (1.2mm,0.5mm)$) arc (65:-245:2.5mm);
   \node[scale=0.75] at ($(n-1-5.south) + (0mm,-6.5mm)$) {$\rho_{n-1}$};
   
     \draw[dash pattern = on 0.5mm off 0.3mm,->] ($(n-1-6.south) +
    (1.2mm,0.5mm)$) arc (65:-245:2.5mm);
   \node[scale=0.75] at ($(n-1-6.south) + (0mm,-6.5mm)$) {$\rho_{n}$};
\end{tikzpicture}\end{equation*}
where the broken arrows are concentrated in degree $-1$ and the remaining generators, i.e.~ solid arrows and idempotents, are in degree $0$. The continuous $k$--linear differential $d\colon \widehat{kQ} \ra \widehat{kQ}$ is completely specified by sending $\rho_{i}$ to the mesh relation (or preprojective relation) starting in the vertex $i$, e.g.~ $d(\rho_{2})=\alpha_{1}\alpha_{1}^*+\alpha_{2}^*\alpha_{2}$.
\end{example*}
 We include a complete list of (the graded quivers, which completely determine) the dg Auslander algebras for ADE--singularities in all Krull dimensions in the Appendix.

\begin{remark}
The triangle equivalence \eqref{E:Cluster-Equiv} and its proof yield relations to generalized cluster categories \cite{Amiot09, Guolingyan11a,ThanhofferdeVolcseyMichelVandenBergh10} and stable categories of special Cohen--Macaulay modules over complete rational surface singularities \cite{Wunram88, IyamaWemyss09, KIWY12}. Moreover, Bridgeland determined a connected component of the stability manifold of $\cd_{fd}\bigl(\Lambda_{dg}(\ul{\MCM}(R))\bigr)$ for ADE-surfaces $R$ \cite{Bridgeland09}. We refer to Section \ref{s:Concluding} for more details on these remarks. 
\end{remark}

\subsection{General properties of relative singularity categories}
In the notations of the setup given in Subsection \ref{ss:SetupIntro} above, we assume that $R$ has an \emph{isolated} singularity and that  $A$ is an NCR of $R$. Let $\ul{A}:=A/AeA \cong\ul{\End}_{R}(M)$ be the corresponding stable endomorphism algebra. Since $R$ is an isolated singularity, $\ul{A}$ is a finite-dimensional $k$-algebra. We denote the simple $\ul{A}$-modules by $S_{1}, \ldots, S_{t}$.  Then the relative singularity category $\Delta_{R}(A)=\cd^b(\mod-A)/K^b(\proj-R)$ has the following properties:
\begin{itemize}
\item[(a)] All morphism spaces are finite-dimensional over $k$ (see \cite
{ThanhofferdeVolcseyMichelVandenBergh10} or Prop. \ref{p:hom-finiteness-of-delta}).
\item[(b)] Let $M_{t+1}, \ldots, M_{s}$ be further indecomposable $\MCM$ $R$-modules and  let $A'=\End_{R}(\bigoplus_{i=0}^s M_{i})$. There exists a fully faithful triangle functor (Prop. \ref{P:Embedding-of-relative})
\begin{align}
\Delta_{R}(A) \longrightarrow \Delta_{R}(A').
\end{align}
\item[(c)] $\Delta_{R}(A)$ is idempotent complete and $K_{0}\big(\Delta_{R}(A)\big)\cong \mathbb{Z}^t$ (see \cite[Thm. 3.2]{BurbanKalck11}).
\item[(d)] There is an exact sequence of triangulated categories (see \cite{ThanhofferdeVolcseyMichelVandenBergh10} or Prop.~\ref{C:From-gen-to-class}) \begin{equation}\label{E:class-as-quotient} \thick(S_{1}, \ldots, S_{t}) = \cd^b_{\ul{A}}(\mod-A) \longrightarrow \Delta_{R}(A) \longrightarrow \cd_{sg}(R), \end{equation}
where $\cd^b_{\ul{A}}(\mod-A)\subseteq \cd^b(\mod-A)$ denotes the full subcategory consisting of complexes with cohomologies in $\mod-\ul{A}$. Moreover, this subcategory admits an intrinsic description inside $\cd^b(\mod-A)$ (see \cite
{ThanhofferdeVolcseyMichelVandenBergh10}  or Cor. \ref{C:Intrinsic}). 
\item[(e)] If $\add M$ has $d$--almost split sequences \cite{Iyama07a}, then $\cd^b_{\ul{A}}(\mod-A)$ has a Serre functor $\nu$, whose action on the generators $S_{i}$ is given by \begin{align}\label{E:IntroFractCY}\nu^n(S_{i}) \cong S_{i}[n(d+1)],\end{align} where $n=n(S_{i})$ is given by the length of the $\tau_{d}$--orbit of $M_{i}$ (Thm.~\ref{t:fractionally-cy-property}).
\item[(f)] In the setup of (e), the sequence in (d) gives a triangle equivalence (Prop.~\ref{R:Orlov})
\begin{align}
\Delta_R(A)_{sg} \cong \cd_{sg}(R),
\end{align} 
where $\Delta_R(A)_{sg}=\Delta_R(A)/\Delta_R(A)_{hf}$ is Orlov's general singularity category construction for arbitrary triangulated categories \cite[Definition 1.7.]{Orlov09}.
\item[(g)] If $\krdim R=3$ and $\MCM(R)$ has a cluster-tilting object $M$, then $C=\End_R(M)$ is a \emph{non-commutative crepant resolution} of $R$, see~\cite[Section 5]{Iyama07}. If $M'$ is another cluster-tilting object in $\MCM(R)$ and $C'=\End_R(M')$, then $?\lten_{C}\Hom_R(M',M)\colon\cd^b(\mod C)\rightarrow\cd^b(\mod C')$ is a triangle equivalence (see \emph{loc. cit.} and~\cite[Prop. 4]{Palu09}), which is compatible with the embeddings from $K^b(\proj R)$~\cite[Cor. 5]{Palu09}. Hence, one obtains a triangle equivalence
\begin{align}
\Delta_R(C)\longrightarrow\Delta_R(C').
\end{align}
\item[(h)] Assume additionally that $R$ has Krull dimension $d \geq 2$ and that $A$ is an NCCR of $R$. Then the subcategory $\Delta_R(A)_{lhf}=\Delta_R(A)_{rhf} \subseteq \Delta_R(A)$ of left (respectively, right) homologically finite objects is $d$-Calabi--Yau. By Proposition \ref{R:Orlov}, this subcategory equals $\thick(\mod-\ul{A}) \subseteq \cd^b(A)$, which is $d$-CY since $\ul{A}$ is finite dimensional, see e.g. \cite[Theorem 4.23]{WemyssLectures}. 

This subcategory can be Calabi--Yau even if $A$ is not an NCCR. For example, let $R_k$ be a $2k$-dimensional ADE-singularity and let $A_k=\Aus(\MCM(R_k))$ be the Auslander algebra. Then $\Delta_{R_k}(A_k) \cong \Delta_{R_1}(A_1)$ for all $k \geq 1$ by our main result (Theorem \ref{t:Classical-versus-generalized-singularity-categories}) and $\Delta_{R_1}(A_1)_{lhf}$ is $2$-CY by the discussion above. But $A_k$ is an NCCR if and only if $k=1$.

\item[(i)] Let $\left(\cd_{\ul{A}}(\Mod-A)\right)^c \subseteq \cd^b_{\ul{A}}(\Mod-A)$ be the full subcategory of compact objects. There is an equivalence of triangulated categories (see Rem. \ref{r:CompInterprOfRelSingCat}).
\begin{align}Ê\Delta_{R}(A) \cong \bigl(\cd_{\ul{A}}(\Mod-A)\bigr)^c. \end{align} 
\end{itemize}

\begin{remark}
 The Hom-finiteness in (a) is surprising since (triangulated) quotient categories tend to behave quite poorly in this respect (see e.g.~\cite[Remark 6.5.]{BurbanKalck11}).
\end{remark}

\begin{remark}
All our results actually hold in the generality of Gorenstein $S$-orders with an isolated singularity, in the sense of Auslander (see \cite[Section III.1]{Auslander76} and \cite{Auslander84}) or finite-dimensional selfinjective $k$-algebras. Here $S=(S, \mathfrak{n})$ denotes a local complete regular Noetherian $k$-algebra, with $k \cong S/\mathfrak{n}$. It is a matter of heavier notation and terminology to generalize our proofs to this setting.
\end{remark}

\subsection{Contents}
Section \ref{s:derived-cat} provides the necessary material on dg algebras and derived categories of dg modules. Moreover, we give an apparently new criterion for the $\Hom$--finiteness of the category of perfect dg modules $\per(B)$ over some non-positive dg $k$-algebra $B$ (Proposition \ref{p:hom-finiteness-of-per}).  
Further, we study recollements of derived categories of dg modules associated to an idempotent. Section \ref{S:Tale} deals with equivalences between triangulated (quotient) categories arising from derived module categories of a right Noetherian ring $A$ and idempotents in $A$. This is used in Section \ref{S:MCM-over-Gor} to show that $\ul{\MCM}(R)$ may be obtained as a triangle quotient of $\Delta_{R}(A)$.
The first result in Section \ref{s:class-to-relative} shows that certain relative singularity categories 
enjoy a (weak) fractional Calabi--Yau property (Theorem \ref{t:fractionally-cy-property}). 
We use this and the results of Section \ref{s:derived-cat}
to show that for `finite' Frobenius categories the relative Auslander singularity categories only depend 
on the stable category (Theorem \ref{t:main-thm-2}). The results from Sections \ref{s:derived-cat}  to \ref{s:class-to-relative} are applied in Section \ref {S:MCM-over-Gor}  to study relative singularity categories over complete isolated Gorenstein $k$-algebras $R$. In particular, we prove our main result (Theorem \ref{t:Classical-versus-generalized-singularity-categories}). In Section \ref{s:Concluding}, we remark on relations to Bridgeland's stability manifold and generalized cluster categories.
In Appendix~\ref{Appendix}, we give a complete list of the dg Auslander algebras for the complex ADE--singularities in all Krull dimensions.

\medskip
\noindent
\emph{Acknowledgement}. The authors would like to thank Hanno Becker, Igor Burban, Will Donovan,
Bernhard Keller, Hao Su, Michel Van den Bergh and Michael Wemyss for helpful discussions on parts of this article. They express their deep gratitude to Bernhard Keller for pointing out an error in a preliminary version and helping to correct it. Moreover, they thank Igor Burban and Michael Wemyss for valuable suggestions concerning the presentation of the material. They thank the referee for helpful comments and suggestions, which led to an improvement of the text. M.K. thanks Stefan Steinerberger for improving the language of this article.
 
 \pagebreak

\section{Derived categories, dg algebras and Koszul duality}\label{s:derived-cat}
\subsection{Notations}\label{ss:notation}
Let $k$ be a commutative ring. Let $D=\Hom_k(?,k)$ denote the
$k$-dual. When the input is a graded $k$-module, $D$ means the
graded dual. Namely, for $M=\bigoplus_{i\in\mathbb{Z}}M^i$, the graded dual $DM$
has components $(DM)^i=\Hom_k(M^{-i},k)$.
\newcommand{\Add}{\opname{Add}}
\subsubsection*{Generating subcategories/subsets}

Let $\ca$ be an additive $k$-category. Let $\cs$ be a subcategory or a subset of objects of $\ca$. We denote by $\add_{\ca}(\cs)$
(respectively, $\Add_{\ca}(\cs)$) the smallest full subcategory of
$\ca$ which contains $\cs$ and which is closed under taking finite
direct sums (respectively, all existing direct sums) and taking
direct summands.

If $\ca$ is a triangulated category, then 
$\thick_{\ca}(\cs)$ (respectively, $\Tria_{\ca}(\cs)$) denotes the smallest
triangulated subcategory of $\ca$ which contains $\cs$ and which is
closed under taking direct summands (respectively, all existing
direct sums). In this situation, the class $\cs$ is called a class of generators of $\thick_\ca(\cs)$. Ê

When it does not cause confusion, we omit the subscripts and write
the above notations as $\add(\cs)$, $\Add(\cs)$, $\thick(\cs)$ and
$\Tria(\cs)$.

\subsubsection*{Derived categories of abelian categories}

Let $\ca$ be an additive $k$-category. 
Let $*\in\{\emptyset,-,+,b\}$ be a boundedness condition. Denote by $K^*(\ca)$ the homotopy
category of complexes of objects in $\ca$ satisfying the boundedness
condition $*$.

Let $\ca$ be an abelian $k$-category. 
Denote by $\cd^*(\ca)$ the derived category of complexes of objects
in $\ca$ satisfying the boundedness condition $*$.

Let $R$ be a $k$-algebra. Without further remark, by an $R$-module
we mean a right $R$-module. Denote by $\Mod-R$ the category of
$R$-modules, and denote by $\mod-R$ (respectively, $\proj-R$) its
full subcategory of finitely generated $R$-modules (respectively,
finitely generated projective $R$-modules). When $k$ is a field, we will also
consider the category $\fdmod-R$ of those $R$-modules which are
finite-dimensional over $k$. We often view
$K^b(\proj-R)$ as a triangulated subcategory of $\cd^*(\Mod-R)$.

\subsubsection*{Truncations}
Let $\ca$ be an abelian $k$-category. For $i\in\mathbb{Z}$ and for a
complex $M$ of objects in $\ca$, we define the \emph{standard
truncations} $\sigma^{\leq i}$ and $\sigma^{>i}$ by 
\begin{align*}(\sigma^{\leq i}M)^j=\begin{cases} M^j & \text{ if
} j<i,\\ \ker d_M^i & \text{ if } j=i,\\ 0 & \text{ if }
j>i,\end{cases} &&&& (\sigma^{>
i}M)^j=\begin{cases} 0 & \text{ if } j<i,\\
{\displaystyle \frac{M^i}{\ker
d_M^i} }& \text{ if } j=i,\\ M^j & \text{ if } j>i,\end{cases}
\end{align*}
and the
\emph{brutal truncations} $\beta_{\leq i}$ and $\beta_{\geq i}$ by
\begin{align*}
(\beta_{\leq i}M)^j=\begin{cases} M^j & \text{ if } j\leq i,\\
0 & \text{ if } j>i,\end{cases} &&&~~
(\beta_{\geq i}M)^j=\begin{cases} 0 & \text{ if } j<i,\\
M^j & \text{ if } j\geq i.\end{cases}
\end{align*}
Their respective differentials are inherited from $M$.
Notice that $\sigma^{\leq i}(M)$ and $\beta_{\geq i}(M)$ are
subcomplexes of $M$ and $\sigma^{>i}(M)$ and $\beta_{\leq i-1}(M)$
are the corresponding quotient complexes. Thus we have two sequences,
which are componentwise short exact,
\begin{align*}
0 \ra \sigma^{\leq i}(M)\ra M\ra \sigma^{>i}(M)\ra 0 & \quad  \text{  and } & 0\ra \beta_{\geq i}(M)\ra  M\ra \beta_{\leq i-1}(M)\ra  0.
\end{align*}
 Moreover, taking standard truncations behaves well with respect to cohomology.
\begin{align*}
H^j(\sigma^{\leq i}M)=\begin{cases} H^j(M) & \text{ if } j\leq
i,\\ 0 & \text{ if } j>i,\end{cases} &&&&
H^j(\sigma^{>i}M)=\begin{cases} 0 & \text{ if } j\leq i,\\
H^j(M) & \text{ if } j>i.\end{cases}
\end{align*}

\subsection{DG algebras and their derived
categories}\label{ss:dg-alg}

Let $A$ be a dg $k$-algebra. Consider the \emph{derived category}
$\cd(A)$ of dg $A$-modules, see~\cite{Keller94}. It is a
triangulated category with shift functor being the shift of
complexes $[1]$. It is obtained from the category $\cc(A)$ of dg $A$-modules
by formally inverting all quasi-isomorphisms. Let $\per(A)=\thick(A_A)$. A $k$-algebra $R$ can be
viewed as a dg algebra concentrated in degree $0$. In this case, we
have $\cd(R)=\cd(\Mod-R)$ and $\per(R)=K^b(\proj-R)$. Assume that
$k$ is a field. Consider the full subcategory $\cd_{fd}(A)$ of
$\cd(A)$ consisting of those dg $A$-modules whose total cohomology
is finite-dimensional. 

Let $M$, $N$ be dg $A$-modules. Define the complex $\cHom_A(M,N)$ componentwise as
\begin{eqnarray*}\cHom_A^i(M,N)=\left.\left\{f\in\prod_{j\in\mathbb{Z}}\Hom_k(M^j,N^{i+j}) \, \right|
\, f(ma)=f(m)a\right\},\end{eqnarray*} with
differential given by $d(f)=d_N\circ f-(-1)^i f\circ d_M$ for
$f\in\cHom_A^i(M,N)$. The complex $\cEnd_A(M)=\cHom_A(M,M)$ with the
composition of maps as product is a dg $k$-algebra. We will use the following results from~\cite{Keller94}. Let $A$ and $B$ be dg $k$-algebras.
\begin{itemize}
 \item[--] Every dg $A$-module $M$ has a natural structure of dg $\cEnd_A(M)$-$A$-bimodule.
 \item[--] If $M$ is a dg $A$-$B$-bimodule, then there is an adjoint pair of triangle functors
\[
\begin{xy}
\SelectTips{cm}{}
\xymatrix{\cd(A)\ar@<.7ex>[rr]^{?\lten_A M}&&\cd(B)\ar@<.7ex>[ll]^{\RHom_B(M,?)}.}
\end{xy}
\]
 \item[--] Let $f:A\rightarrow B$ be a quasi-isomorphism of dg algebras. Then the induced triangle functor $?\lten_A B:\cd(A)\rightarrow
\cd(B)$ is an equivalence. A quasi-inverse is given by the restriction $\cd(B)\rightarrow\cd(A)$ along $f$. It can be 
written as $?\lten_B B=\RHom_B(B,?)$ where $B$ is considered as a dg $B$-$A$-bimodule respectively dg $A$-$B$-bimodule via $f$.
These equivalences restrict to equivalences between $\per(A)$ and $\per(B)$ and, if $k$ is a field, between $\cd_{fd}(A)$ and $\cd_{fd}(B)$.
By abuse of language, by a quasi-isomorphism we will also mean a zigzag of quasi-isomorphisms.
\end{itemize}

\subsection{The Nakayama functor}\label{ss:nakayama-functor} We follow \cite[Section 10]{Keller94}.
Let $k$ be a field and let $A$ be a dg $k$-algebra. We
consider the dg functor $\nu=\nu_A=D\cHom_A(?,A)\colon \cc(A)\rightarrow\cc(A)$.  
By
abuse of notation the left derived functor of $\nu$ is also denoted by $\nu$. It is clear that
$\nu(A)=D(A)$ holds, so $\nu$ restricts to a triangle functor
\[ 
\nu\colon\per(A)\rightarrow\thick(D(A)),
\]
which is a triangle equivalence provided that $A$ has finite-dimensional cohomologies in each degree.

Furthermore, for dg $A$-modules $M$ and $N$ there is a binatural map
\begin{align}\label{E:Naka}
\begin{array}{c}
\begin{xy}
\SelectTips{cm}{}
\xymatrix@R=0.5pc{D\cHom_A(M,N)\ar[r] & \cHom_A(N,\nu(M))\\
\varphi\ar@{|->}[r]&\bigl(n\mapsto (f\mapsto \varphi(g))\bigr)}
\end{xy}
\end{array}
\end{align} where
$f\in\cHom_A(M,A)$ and $g\colon m\mapsto nf(m)$.  If we let $M=A$, then $(\ref{E:Naka})$ is an isomorphism and hence a quasi-isomorphism for $M\in\per(A)$. Passing to the derived category in $(\ref{E:Naka})$ yields a binatural isomorphism for $M\in\per(A)$ and $N\in\cd(A)$:
\begin{align}\label{E:AR-formula}
\begin{xy}
\SelectTips{cm}{}
\xymatrix{D\Hom_{\cd(A)}(M,N)\cong\Hom_{\cd(A)}(N,\nu(M))}
\end{xy}\! \! .
\end{align}
This shows that $\nu$ defines a (right) Serre functor on $\per(A)=K^b(\proj-A)$, if $A$
is a finite-dimensional Gorenstein $k$-algebra.

\subsection{Non-positive dg algebras: $t$-structures and co-$t$-structures}\label{ss:nonpositive-dg-alg-1}

Let $\cc$ be a triangulated $k$-category with shift functor $[1]$. A
\emph{$t$-structure} on $\cc$ \cite{BeilinsonBernsteinDeligne82}
is a pair $(\cc^{\leq 0},\cc^{\geq 0})$ of strictly (i.e.~ closed under isomorphisms) full
subcategories such that
\begin{itemize}
\item[--] $\cc^{\leq 0}[1]\subseteq\cc^{\leq 0}$ and
$\cc^{\geq 0}[-1]\subseteq\cc^{\geq 0}$,
\item[--] $\Hom(M,N[-1])=0$ for $M\in\cc^{\leq 0}$
and $N\in\cc^{\geq 0}$,
\item[--] for each $M\in\cc$ there is a triangle $M'\rightarrow
M\rightarrow M''\rightarrow M'[1]$ in $\cc$ with $M'\in\cc^{\leq 0}$
and $M''\in\cc^{\geq 0}[-1]$.
\end{itemize}
The \emph{heart} $\cc^{\leq 0}\cap\cc^{\geq 0}$ of the $t$-structure $(\cc^{\leq 0},\cc^{\geq 0})$ is an abelian category \cite{BeilinsonBernsteinDeligne82}.

Let $A$ be a dg $k$-algebra such that $A^i=0$ for $i>0$. Such a
dg algebra is called a \emph{non-positive dg algebra}. The canonical projection $A\rightarrow H^0(A)$ is a homomorphism of
dg algebras. We view a module over $H^0(A)$ as a dg module over $A$ via this homomorphism.
This defines a natural functor $\Mod-H^0(A)\rightarrow \cd(A)$.

\begin{proposition}\label{p:standard-t-str}
Let $A$ be a non-positive dg algebra.
\begin{itemize}
\item[(a)] \emph{(\cite[Theorem 1.3]{HoshinoKatoMiyachi02}, \cite[Section
2.1]{Amiot09} and \cite[Section 5.1]{KellerYang11})} Let $\cd^{\leq
0}$ respectively $\cd^{\geq 0}$ denote the full subcategory of
$\cd(A)$ which consists of objects $M$ such that $H^i(M)=0$ for
$i>0$ respectively for $i<0$. Then $(\cd^{\leq 0},\cd^{\geq 0})$ is
a $t$-structure on $\cd(A)$. Moreover, taking $H^0$ is an
equivalence from the heart to $\Mod-H^0(A)$, and the natural functor $\Mod-H^0(A)\rightarrow \cd(A)$
induces a quasi-inverse to this equivalence. We will identify $\Mod-H^0(A)$ with the heart via these equivalences.
\item[(b)] Assume that $k$ is a field. The $t$-structure in (a) restricts to a $t$-structure on
$\cd_{fd}(A)$ whose heart is $\fdmod-H^0(A)$.
Moreover, as a triangulated category $\cd_{fd}(A)$ is generated by the heart.
\item[(c)] Assume that $k$ is a field. Assume that $\cd_{fd}(A)\subseteq\per(A)$ and $\per(A)$
is Hom-finite. Then the $t$-structure in (a) restricts to a
$t$-structure on $\per(A)$, whose heart is $\fdmod-H^0(A)$.
\end{itemize}
\end{proposition}
\begin{proof} (a) We only give the key point. Let $M$ be a dg
$A$-module. Thanks to the assumption that $A$ is non-positive, the
standard truncations $\sigma^{\leq 0}M$ and $\sigma^{>0}M$ are again
dg $A$-modules. A dg $A$-module $M$ whose cohomologies are
concentrated in degree $0$ is related to the  $H^0(A)$-module
$H^0(M)$ via the following chain of quasi-isomorphisms 
 $
\begin{tikzpicture}[description/.style={fill=white,inner sep=0pt}]
    \matrix (n) [matrix of math nodes, row sep=0em,
                 column sep=2 em, text height=12pt, text depth=0ex,
                 inner sep=0pt, nodes={inner xsep=0em, inner
ysep=0em}]
    {  
       H^0(M)&\sigma^{\leq 0}M &M. \\
    };
 \draw[->>] ($(n-1-2.west) + (-0.6mm,-0.8mm)$) -- ($(n-1-1.east) + (0mm,-0.8mm)$);
 \draw[<-right hook] ($(n-1-3.west) + (-0.6mm,-0.8mm)$) -- ($(n-1-2.east) + (0mm,-0.8mm)$);    
\end{tikzpicture}    
$
Moreover, we have a distinguished triangle
\begin{eqnarray}\label{e:triangle-standard-truncation} \sigma^{\leq
0}M\longrightarrow M\longrightarrow \sigma^{>0}M\longrightarrow \sigma^{\leq
0}M[1]\end{eqnarray} in $\mathcal{D}(A)$. This is the triangle required to show that
$(\cd^{\leq 0},\cd^{\geq 0})$ is a $t$-structure.

(b) For the first statement, it suffices to show that, under the
assumptions, the standard truncations are endo-functors of
$\cd_{fd}(A)$. This is true because $H^*(\sigma^{\leq 0}M)$ and
$H^*(\sigma^{>0}M)$ are subspaces of $H^*(M)$.

To show the second statement, let $M\in\cd_{fd}(M)$. Suppose that
for $m\geq n$ we have $H^{n}(M)\neq 0$, $H^{m}(M)\neq 0$ but
$H^i(M)=0$ for $i>m$ or $i<n$. We prove that $M$ is generated by the
heart by induction on $m-n$. If $m-n=0$, then a shift of $M$ is in the heart.
Now suppose $m-n>0$. The standard truncations yield a triangle
\[
\sigma^{\leq n}M \longrightarrow M \longrightarrow \sigma^{>n}M \longrightarrow \sigma^{\leq n}M[1].
\]
Now the cohomologies of $\sigma^{\leq n}M$ are concentrated in degree
$n$, and hence $\sigma^{\leq n}M$ belongs to a shifted copy of the heart.
Moreover, the cohomologies of $\sigma^{>n}(M)$ are bounded between
degrees $n+1$ and $m$. By induction hypothesis $\sigma^{>n}(M)$ is
generated by the heart. Therefore $M$ is generated by the heart.

(c) Same as the proof for~\cite[Proposition 2.7]{Amiot09}.
\end{proof}

Let $\cc$ be as above. A
\emph{co-$t$-structure} on $\cc$~\cite{Pauksztello08}  (or
\emph{weight structure}~\cite{Bondarko10}) is a pair $(\cc_{\geq
0},\cc_{\leq 0})$ of strictly full subcategories of $\cc$ satisfying the following conditions
\begin{itemize}
\item[(C1)] both $\cc_{\geq 0}$ and $\cc_{\leq
0}$ are closed under finite direct sums and direct summands,
\item[(C2)] $\cc_{\geq 0}[-1]\subseteq\cc_{\geq 0}$ and
$\cc_{\leq 0}[1]\subseteq\cc_{\leq 0}$,
\item[(C3)] $\Hom(M,N[1])=0$ for $M\in\cc_{\geq 0}$
and $N\in\cc_{\leq 0}$,
\item[(C4)] for each $M\in\cc$ there is a triangle $M'\rightarrow
M\rightarrow M''\rightarrow M'[1]$ in $\cc$ with $M'\in\cc_{\geq 0}$
and $M''\in\cc_{\leq 0}[1]$.
\end{itemize}
It follows that $\cc_{\leq 0}=\cc_{\geq 0}^\perp[-1]$. The
\emph{co-heart} is defined as the intersection $\cc_{\geq
0}~\cap~\cc_{\leq 0}$.

\begin{lemma}\label{l:stupid-truncation}
 \emph{(\cite[Proposition 1.3.3.6]{Bondarko10})} For $M\in\cc_{\leq 0}$, there exists a distinguished triangle
$M' \ra M\ra M''\ra M'[1]$ with $M'\in\cc_{\geq 0}\cap\cc_{\leq 0}$ and $M''\in\cc_{\leq
0}[1]$.
\end{lemma}
\begin{proof}
For convenience, we include the proof. Take a triangle as in (C4). Since $\cc_{\leq 0}=\cc_{\geq 0}^{\perp}[-1]$ is extension closed, $M'$ is in $\cc_{\leq 0}$. Indeed $M'$ is an extension of $M''[-1]$ and $M$.
\end{proof}

\noindent Let $A$ be a non-positive dg $k$-algebra. Let $\cp_{\geq 0}$
respectively $\cp_{\leq 0}$ denote the smallest full subcategory of
$\per(A)$ which contains $A[i]$ for $i\leq 0$ respectively
$i\geq 0$ and is closed under taking extensions and direct
summands. We need a result from \cite{Bondarko10}.

\begin{proposition}\label{p:standard-co-t-str}
$(\cp_{\geq 0},\cp_{\leq 0})$ is a co-$t$-structure of $\per(A)$, with co-heart $\add(A)$.
\end{proposition}
\begin{proof}
 This follows from~\cite[Proposition 5.2.2, Proposition 6.2.1]{Bondarko10}, see also \cite[Proposition 3.4]{IyamaYang14} for an elementary proof (note that $A$ is a silting object in $\per(A)$).
\end{proof}

\noindent Hence, objects in $\cp_{\leq 0}$ are characterised by
the vanishing of the positive cohomologies:
\begin{corollary}\label{c:description-of-coaisle-by-cohomologies}
 $\cp_{\leq 0}=\{M\in\per(A)\, | \, H^i(M)=0\text{ for any } i>0\}$.
\end{corollary}
\begin{proof} Let $\cs$ be the category on the right.
 By the preceding proposition, $\cp_{\leq 0}=\cp_{\geq 0}^{\perp}[-1]=(\cp_{\geq 0}[-1])^\perp$.
In particular, for $M\in\cp_{\leq 0}$ and $i<0$ this implies that 
$\Hom(A[i],M)=0$. Hence, $H^i(M)=0$ holds for any
$i>0$ and $M$ is in $\cs$. Conversely, if $H^i(M)=\Hom(A[-i],M)=0$
for any $i>0$, then it follows by d\'evissage that $\Hom(N,M)=0$ for any $N\in\cp_{\geq
0}[-1]$. This shows that $M$ is contained in $\cp_{\leq 0}$.
\end{proof}

\subsection{Non-positive dg algebras: Hom-finiteness}\label{ss:nonpositive-dg-alg-2}

Let $A$ be a dg algebra. Then the subcomplex $\sigma^{\leq 0}A$ inherits a dg algebra structure from $A$.
Therefore if $H^i(A)=0$ for any $i>0$, the embedding $\sigma^{\leq 0}A\hookrightarrow A$ is a quasi-isomorphism of dg algebras.

We generalise~\cite[Lemma 2.5 \& Prop. 2.4]{Amiot09}
and~\cite[Lemma 2.4 \& Prop. 2.5]{Guolingyan11a}.

\begin{proposition}\label{p:hom-finiteness-of-per}
 Let $k$ be a field and $A$ be a dg $k$-algebra such that
\begin{itemize}
 \item[--] $A^i=0$ for any $i>0$,
 \item[--] $H^0(A)$ is finite-dimensional,
 \item[--] $\cd_{fd}(A)\subseteq\per(A)$.
\end{itemize}
Then $H^i(A)$ is finite-dimensional for any $i$. Consequently,
$\per(A)$ is Hom-finite.
\end{proposition}
\begin{proof}
 It suffices to prove the following induction step: 
 if $H^i(A)$ is finite-dimensional for $-n\leq i\leq 0$, then
$H^{-n-1}(A)$ is finite-dimensional.

To prove this claim, we consider the triangle induced by the standard truncations 
\[
\sigma^{\leq -n-1}A \longrightarrow A\longrightarrow \sigma^{> -n-1}A\longrightarrow (\sigma^{\leq -n-1}A)[1].
\]
Since $H^i(\sigma^{> -n-1}A)=H^i(A)$ for $i\geq -n$, it follows by
the induction hypothesis that $\sigma^{> -n-1}A$ belongs to
$\cd_{fd}(A)$, and hence to $\per(A)$ by the third assumption on
$A$. Therefore $\sigma^{\leq -n-1}A\in\per(A)$. By
Corollary~\ref{c:description-of-coaisle-by-cohomologies},
$(\sigma^{\leq -n-1}A)[-n-1]\in \cp_{\leq 0}$. Moreover,
Lemma~\ref{l:stupid-truncation} and
Proposition~\ref{p:standard-co-t-str} imply that there is a triangle
\[
M' \longrightarrow (\sigma^{\leq -n-1}A)[-n-1] \longrightarrow M'' \longrightarrow M'[1]
\]
with $M'\in\add(A)$ and $M''\in\cp_{\leq 0}[1]$. It follows from
Corollary~\ref{c:description-of-coaisle-by-cohomologies} that
$H^0(M'')=0$. Therefore applying $H^0$ to the above triangle, we
obtain an exact sequence
\[
H^0(M')\longrightarrow  H^0((\sigma^{\leq -n-1}A)[-n-1])=H^{-n-1}(A) \longrightarrow 0.
\]
But $H^0(M')$ is finite-dimensional because $M'\in\add(A)$ holds and
$H^0(A)$ has finite dimension by assumption. Thus $H^{-n-1}(A)$ is finite-dimensional.
\end{proof}

\subsection{Complete path algebras and minimal relations}\label{ss:minimal-relation}

Let $k$ be a field and $Q$ be a finite quiver. Denote by $\widehat{kQ}$ the \emph{complete path algebra} of $Q$, i.e.~
the completion of the path algebra $kQ$ with respect to the $\mathfrak{m}$-adic topology, where $\mathfrak{m}$ is the ideal
of $kQ$ generated by all arrows. Namely, $\widehat{kQ}$ is the inverse limit in the category of algebras of the inverse system $\{kQ/\mathfrak{m}^n,
\pi_n:kQ/\mathfrak{m}^{n+1}\rightarrow kQ/\mathfrak{m}^n\}_{n\in\mathbb{N}}$, where $\pi_n$ is the canonical projection.
Later we will also work with completions $\widehat{kQ}$ of path algebras of graded quivers $Q$: they are defined as above with the inverse limit taken in the category of graded algebras. In particular, the degree $n$ component of $\widehat{kQ}$ consists of elements of the form $\sum_p \lambda_p p$ where the sum is taken over all paths of degree $n$ and $\lambda_p \in k$.

The complete path algebra $\widehat{kQ}$ has a natural topology, the $J$-adic topology for $J$ the ideal
generated by all arrows. Let $I$ be a closed ideal of $\widehat{kQ}$ contained in $J^2$ and let $A=\widehat{kQ}/I$. For a vertex $i$ of $Q$, let $e_i$ denote the trivial path at $i$. A \emph{set of minimal relations} of $A$ (or of $I$) is a finite subset $R$ of $\bigcup_{i,j\in Q_0}e_iIe_j$ such that  $I$ coincides with the closure $\overline{(R)}$ of the ideal of $\widehat{kQ}$ generated by $R$  but not with $\overline{(R')}$ for any proper subset $R'$ of $R$.

\newcommand{\rad}{\opname{rad}}
\subsection{Koszul duality}\label{ss:dual-bar-construction} In this subsection we recall the dual bar
construction of an $A_\infty$-algebra and discuss  Koszul duality for locally finite non-positive dg algebras. Our main references are~\cite{Keller06c,Lefevre03,LuPalmieriWuZhang04,LuPalmieriWuZhang08}. Our main aim is to establish Theorem~\ref{t:koszul-duality-for-non-positive-dg-alg}, which plays an essential role in the proof of Theorem~\ref{t:main-thm-2}. 

\medskip

We start with the definition of $A_\infty$-algebras.
Assume that $k$ is a field. Let $K$ be a direct product of finitely many copies of $k$ and consider it as a $k$-algebra via the diagonal embedding.

An \emph{$A_\infty$-algebra} $A$ over $K$ is a graded $K$-bimodule endowed with a family of homogenous $K$-bilinear maps of degree
$1$ (called \emph{multiplications}) $\{b_n:(A[1])^{\ten_K
n}\rightarrow A[1]|n\geq 1\}$ satisfying the following identities
\begin{eqnarray}\label{e:multiplications}
\sum_{i+j+l=n}b_{i+1+l}(id^{\ten i}\ten b_j\ten id^{\ten l})=0,
\qquad n\geq 1.
\end{eqnarray}
Here we take a non-commutative viewpoint: for an $A_\infty$-algebra over $K$ the left and right graded $K$-module structures may be different, while for an $A_\infty$-$K$-algebra in the usual sense (see \eg \cite[Section 2.1]{Lunts10}) the left and right graded $K$-module structures coincide.  Let $A$ be an $A_\infty$-algebra over $K$. $A$ is said to be \emph{strictly unital} if there is a $K$-bilinear map $\eta: K\to A$ (called the \emph{unit} of $A$) which is homogeneous of degree $0$ such that $b_n(id\ten\cdots\ten id\ten \eta\ten id\ten \cdots\ten id)=0$ for $i\neq 2$ and $b_2(id_A\ten\eta)=b_2(\eta\ten id_A)=id_A$.
Note that the identity \eqref{e:multiplications} for $n=1$ is $b_1^2=0$, thus $A$ is a complex of $K$-bimodules  with differential $b_1$. $A$ is said to be \emph{minimal} if $b_1=0$. A graded $K$-sub-bimodule $B$ of $A$ is called an \emph{$A_\infty$-subalgebra} if for all $n\geq 1$ and for all $a_1,\ldots,a_n\in B$ we have $b_n(a_1\ten\cdots\ten a_n)\in B$. A dg algebra $A$ over $K$ is a strictly unital $A_\infty$-algebra over $K$ with vanishing $b_n$ for $n\geq 3$. The differential of $A$ is $d=-b_1$ and the multiplication is $a_1a_2=(-1)^{|a_1|}b_2(a_1\ten a_2)$ for elements $a_1,a_2$ of $A$ such that $a_1$ is homogeneous of degree $|a_1|$. We emphasize that the notion of dg algebra over $K$ and the notion of dg $K$-algebra (which is $A_\infty$-$K$-algebra with vanishing $b_n$ for $n\geq 3$) are different. 
We remark that shifting the $b_n$'s properly yields another family of maps $\{m_n:A^{\ten_K n}\rightarrow A|n\geq 1\}$, which contains the same information as the family $\{b_n|n\geq 1\}$ and is often used to define $A_\infty$-algebras in the literature.

Let $A$ and $B$ be two strictly unital $A_\infty$-algebras  over $K$. A \emph{morphism} $f:A\rightarrow B$ of strictly unital $A_\infty$-algebras is a family $\{f_n:(A[1])^{\ten_K n}\rightarrow B[1]\}$ of homogeneous $K$-bilinear maps of degree $0$  such that $f_1\eta_A[1]=\eta_B[1]$, $f_n(id\ten\cdots \ten id\ten\eta_A[1]\ten id\ten \cdots\ten id)=0$ for all $n\geq 2$, and that
\begin{eqnarray}\label{e:morphisms}
\sum_{i+j+l=n}f_{i+1+l}(id^{\ten i}\ten b_j\ten id^{\ten l})=\sum_{i_1+\ldots+i_s=n}b_s(f_{i_1}\ten \cdots\ten f_{i_s}),
\qquad n\geq 1.
\end{eqnarray}
It follows that $f_1$ is a chain map with respect to the differentials $b_1$. If $f_1$ is a quasi-isomorphism of complexes, we say that $f$ is an \emph{$A_\infty$-quasi-isomorphism}.
If $f_n=0$ for $n\geq 2$, then the above identities amounts to saying that $f_1:A[1]\rightarrow B[1]$ commutes with all multiplications $b_n$. In this case, we shall call the corresponding map $f_1[-1]:A\rightarrow B$ a \emph{map} of strictly unital $A_\infty$-algebras. Further, it is an \emph{isomorphism} of $A_\infty$-algebras if it is in addition an isomorphism of graded $K$-bimodules.

Let $A$ be a strictly unital $A_\infty$-algebra over $K$. $A$ is said to be \emph{augmented} if there is a map $\varepsilon: A\to K$ of strictly unital $A_\infty$-algebras, which is called the \emph{augmentation} of $A$. Let $A$ and $B$ be two augmented $A_\infty$-algebras over $K$. A \emph{morphism} $f: A\to B$ of augmented $A_\infty$-algebras is a morphism of strictly unital $A_\infty$-algebras such that $\varepsilon_B f_1[-1]=\varepsilon_A$.

Let $A$ be an augmented $A_\infty$-algebra over $K$.  Then the cohomology $H^*(A)$ carries the structure of a minimal augmented $A_\infty$-algebra over $K$ and there is an $A_\infty$-quasi-isomorphism $i: H^*(A)\to A$ of augmented $A_\infty$-algebras such that $i_1$ induces the identity on cohomologies. See \cite{Kadeishvili80}, \cite[Corollaires 1.4.1.4 and 3.2.4.1]{Lefevre03} and \cite{SuHao16}. We call it the \emph{minimal model} of $A$.

\medskip

Let $A$ be an augmented $A_\infty$-algebra over $K$. Denote by $\bar{A}=\ker\varepsilon$. Note that $\bar{A}$ is an $A_\infty$-subalgebra of $A$.
The \emph{bar construction} of $A$, denoted by $BA$, is the graded $K$-bimodule
$$T_{K}(\bar{A}[1])=K\oplus \bar{A}[1]\oplus\bar{A}[1]\ten_{K}\bar{A}[1]\oplus\ldots.$$
It is naturally a coalgebra with comultiplication defined by
splitting the tensors. Moreover, $\{b_n|n\geq 1\}$ uniquely extends
to a $K$-bilinear differential on $BA$ making it a dg coalgebra over $K$. We define the 
\emph{Koszul dual} of $A$ as the graded $k$-dual of $BA$:
\[E(A)=B^{\#}A:=D(BA).\]
As a graded algebra $E(A)=\hat{T}_{K}(D(\bar{A}[1]))$ is the complete tensor algebra of $D(\bar{A}[1])=\Hom_k(\bar{A}[1],k)$
over $K$. It is naturally an augmented dg algebra over $K$ with differential $d$ being the unique continuous $K$-bilinear map satisfying the graded Leibniz rule and taking 
$f \!\in D(\bar{A}[1])$ to $d(f)\! \in B^{\#}\!A$,~defined by
\[d(f)(a_1\ten\cdots \ten a_n)=f(b_n(a_1\ten \cdots \ten a_n)),~~a_1,\ldots,a_n\in \bar{A}[1].\]
Let $\mathfrak{m}$ be the ideal of $E(A)$ generated by $D(\bar{A}[1])$. Then $A$ being minimal amounts to saying that
$d(\mathfrak{m})\subseteq \mathfrak{m}^2$ holds true.  The minimal model $A^*$ of $E(A)$ is called 
the \emph{$A_\infty$-Koszul dual}
of $A$.

The following result is known, see for example \cite{SuHao16} and \cite[Section 19 exercise 4]{FelixHalperinThomas01}. An $A_\infty$-algebra over $K$ has the natural structure of an $A_\infty$-$k$-algebra. Similarly, a dg algebra $A$ over $K$ has the natural structure of a dg $k$-algebra. When considering dg $A$-modules and derived categories, we use this structure of $A$.

\begin{lemma}\label{l:koszul-vs-koszul}
If $A$ is an augmented dg algebra over $K$, then $E(A)$ is quasi-isomorphic to $\RHom_{A}(K,K)$ as a dg $k$-algebra. Here we consider $K$ as a dg $A$-module via the augmentation $\varepsilon$.  In particular, $A^*$ is isomorphic to $\bigoplus_{i\in\mathbb{Z}}\Hom_{\cd(A)}(K,\Sigma^i K)$ as a graded $k$-algebra.
\end{lemma}

If there is an $A_\infty$-quasi-isomorphism  $A\to B$ of augmented $A_\infty$-algebras over $K$, then $E(A)$ and $E(B)$ are quasi-isomorphic as augmented dg algebras over $K$,  see for example \cite[Proposition 1.3.5.1.b and proof of Corollaire 1.3.1.3.b]{Lefevre03}.

\begin{theorem}
\label{t:koszul-double-dual} 
If $E(A)$ is finite-dimensional in each degree,
then $E(E(A))$ is $A_\infty$-quasi-isomorphic to $A$ as an augmented $A_\infty$-algebra over $K$. In particular, $A$ is $A_\infty$-quasi-isomorphic to $E(A^*)$ as an augmented $A_\infty$-algebra over $K$.
\end{theorem}
\begin{proof} The first statement is a `several-object' version of 
 {\cite[Theorem 2.4]{LuPalmieriWuZhang08}} with trivial Adams grading. The proof of  {\cite[Theorem 2.4]{LuPalmieriWuZhang08}} can be easily adapted here.
\end{proof}

Moreover, we can describe the graded algebra underlying $E(A)$ in terms of quivers. Let $e_1,\ldots,e_r$ be the standard basis of $K$. Define $Q$ as the following graded quiver:
\begin{itemize}
\item[--] the set of vertices is $\{1,\ldots,r\}$, 
\item[--] the set of arrows of degree $m$ from $i$ to $j$ is given by a $k$-basis of the degree $m$ component of $e_jD(\bar{A}[1])e_i$. 
\end{itemize}
Then
as a graded algebra $E(A)$ is the graded completion $\widehat{kQ}$ (with respect to the ideal generated by the arrows) of the path algebra $kQ$ of the graded quiver $Q$, see Subsection \ref{ss:minimal-relation} for more details on graded completions.
In this case, $A$ being minimal amounts to saying that the differential of $E(A)$ takes an arrow of $Q$ to a possibly infinite linear combination of paths of length $\geq 2$. In particular, if $A$ is minimal and is concentrated in non-negative degrees and the degree $0$ component of $A$ is $K$, then $Q$ is concentrated in non-positive degrees and the quiver of the algebra $H^0(E(A))$ is the degree zero component $Q^0$ of $Q$.

\smallskip

\noindent{\bf Koszul duality for locally finite non-positive dg algebras.}
The following special case will be important in Section~\ref{s:class-to-relative}. Let $B$ be an augmented dg algebra over $K$ which is a non-positive dg $k$-algebra in the sense of Subsection~\ref{ss:nonpositive-dg-alg-1} such that $H^i(B)$ is finite-dimensional for all $i\in\mathbb{Z}$.  We consider $K$ as a dg $B$-module via the augmentation $\varepsilon$ and denote it by $S$.

\begin{theorem}\label{t:koszul-duality-for-non-positive-dg-alg}
Keep the notation and assumption in the preceding paragraph. Then
\begin{itemize}
\item[(a)] $B^*$ is $A_\infty$-quasi-isomorphic to $\RHom_{B}(S,S)$ as an $A_\infty$-$k$-algebra. 
\item[(b)] $B$ is quasi-isomorphic to $E(B^*)$ as a dg $k$-algebra. 
\end{itemize}
\end{theorem}
\begin{proof}
(a) The desired result follows immediately from Lemma~\ref{l:koszul-vs-koszul} as $B^*$ is $A_\infty$-quasi-isomorphic to $E(B)$.

(b) All the ($A_\infty$-)quasi-isomorphisms below respect  augmentations over $K$. Let $A$ be the minimal model of $B$. Then $E(A)$ and $E(B)$ are quasi-isomorphic, and $A^*$ and $B^*$ are $A_\infty$-isomorphic. As $A$ is non-positive and is finite-dimensional in each degree, $E(A)$ is finite-dimensional in each degree by construction. Applying Theorem~\ref{t:koszul-double-dual} to $A$, we see that $A$ is $A_\infty$-quasi-isomorphic to $E(A^*)$, which is quasi-isomorphic to $E(B^*)$. Therefore $B$ is $A_\infty$-quasi-isomorphic to $E(B^*)$, in particular, they are $A_\infty$-quasi-isomorphic as $A_\infty$-$k$-algebras, and the desired result follows immediately from \cite[Proposition 2.8]{Lunts10}.
\end{proof}

\subsection{Recollements generated by idempotents}\label{ss:recollement}

In this subsection our object of study is the triangle quotient $K^b(\proj-A)/\thick(eA)$, where $A$ is an algebra and $e\in A$
is an idempotent. By Keller's Morita theorem for triangulated categories~\cite[Theorem 3.8 b)]{Keller06d}, the idempotent completion
of this category is equivalent to the perfect derived category $\per(B)$ of some dg algebra $B$. Below we show that we can choose $B$ such that there is a homomorphism of dg algebras $A\rightarrow B$, the restriction
$\cd(B)\rightarrow\cd(A)$ along which is fully faithful.

Following~\cite{BeilinsonBernsteinDeligne82}, a \emph{recollement} of triangulated categories is a diagram
\begin{align}
\begin{xy}
\SelectTips{cm}{}
\xymatrix{\ct''\ar[rr]|{i_*=i_!}&&\ct\ar[rr]|{j^!=j^*}\ar@/^15pt/[ll]^{i^!}\ar@/_15pt/[ll]_{i^*}&&\ct'\ar@/^15pt/[ll]^{j_*}\ar@/_15pt/[ll]_{j_!}}
\end{xy}
\end{align}
of
triangulated categories and triangle functors
such that
\begin{itemize}
\item[1)] $(i^*,i_*=i_!,i^!)$ and
$(j_!,j^!=j^*,j_*)$ are adjoint triples;
\item[2)] $j_!,i_*=i_!,j_*$ are fully faithful;
\item[3)] $j^* i_*=0$;
\item[4)] for every object $X$ of $\ct$ there exist two distinguished triangles
\[i_!i^!X \rightarrow X \rightarrow j_*j^*X \rightarrow i_!i^!X[1] \, \,
\text{ and } \, \,  j_!j^!X \rightarrow X \rightarrow
i_*i^*X \rightarrow j_!j^!X[1],\] where the morphisms starting from and
ending at $X$ are the units and counits.
\end{itemize}


\begin{remark}\label{r:identity-units-and-counits}
Conditions 1) and 2) in conjunction with the Yoneda lemma imply that the composites $i^*i_*$, $i^!i_!$, $j^!j_!$ and $j^*j_*$ are equivalent to the identity functors. 
\end{remark}
 
Let $k$ be a commutative ring and $A$
be a $k$-algebra. 
The following Proposition shows that every idempotent $e \in A$ gives rise to a recollement.
In the literature, much attention has been paid to the special case that $B$ has cohomologies concentrated in degree $0$, see for example~\cite{Cline-Parshall-Scott88,ClineParshallScott96,ClineParshallScott97,KoenigNagase09}.
Recall that $A$ can be viewed as a dg $k$-algebra concentrated in degree $0$ and in this case $\cd(A)=\cd(\Mod-A)$.

\begin{proposition}\label{p:recollement-from-projective-general-case}
Let $A$ be a flat $k$-algebra and $e\in A$ an idempotent. There is a
dg $k$-algebra $B$ with a homomorphism of dg $k$-algebras
$f\colon A\rightarrow B$ and a recollement of derived categories
\begin{align}\label{E:RecollmentByIdempotent}
\begin{xy}
\SelectTips{cm}{}
\xymatrix{\cd(B)\ar[rr]|{\, \, i_*=i_! \, \,}&&\cd(A)\ar[rr]|{\, \,  j^!=j^* \, \, }\ar@/^15pt/[ll]|{\, \, i^! \, \,}\ar@/_15pt/[ll]|{\, \, i^*\, \, }&&\cd(eAe)\ar@/^15pt/[ll]|{\, \, j_* \, \,}\ar@/_15pt/[ll]|{\, \, j_! \, \,}}
\end{xy},
\end{align}
such that the following conditions are satisfied
\begin{itemize}
 \item[(a)] the
adjoint triples $(i^*,i_*=i_!,i^!)$ and $(j_!,j^!=j^*,j_*)$ are given by
\[\begin{array}{ll}
i^*=?\lten_A B, & j_!=?\lten_{eAe} eA,\\
i_*=\RHom_{B}(B,?), &
j^!=\RHom_{A}(eA,?),\\
i_!=?\lten_{B}B, & j^*=?\lten_A Ae,\\
i^!=\RHom_A(B,?),& j_*=\RHom_{eAe}(Ae,?),
\end{array}\]
where $B$ is considered as an $A$-$A$-bimodule via the morphism $f$;
 \item[(b)] $B^i=0$ for $i>0$;
 \item[(c)] $H^0(B)$ is isomorphic to $A/AeA$.
\end{itemize}
\end{proposition}

This result is known to hold in greater generality, see~\cite[Section 2 and 3]{Dwyer02} (which uses different terminologies). For convenience, we include a proof.

\begin{proof}
The exact functor $\Hom_{A}(eA, ?)=? \otimes_{A}Ae \colon \Mod-A \ra \Mod-eAe$ has both a left adjoint $? \otimes_{eAe}eA$ and a right adjoint $\Hom_{eAe}(Ae, ?)$. Deriving these functors, we obtain the right half of the recollement \eqref{E:RecollmentByIdempotent}.  The derived functors are still adjoint and it is known that $? \lten_{eAe}eA$ is fully faithful (see e.g.~\cite[Lemma 4.2]{Keller94}). An application of the Yoneda Lemma shows that there is a natural isomorphism $j^!j_{!} \cong 1_{\cd(eAe)}$. Hence, $j^!$ is a quotient functor with left and right adjoints. In particular, it is a so called Bousfield localization and colocalization functor. Its kernel is $\cd_{A/AeA}(A) \subseteq \cd(A)$, the full subcategory of complexes with cohomologies in $\Mod-A/AeA$. Hence, $j^!$ yields a recollement (see e.g.~\cite[Section 9]{Neeman99})  

\begin{align}\label{E:FirstStepRecollmentByIdempotent}
\begin{xy}
\SelectTips{cm}{}
\xymatrix{\cd_{A/AeA}(A)\ar[rr] &&\cd(A)\ar[rr]|{\, \,  j^!=j^* \, \, } \ar@/_15pt/[]!<0ex, 0ex>;[ll]!<3ex, 0ex> \ar@/^15pt/[]!<0ex, 0ex>;[ll]!<3ex, 0ex> &&\cd(eAe)\ar@/^15pt/[ll]|{\, \, j_* \, \,}\ar@/_15pt/[ll]|{\, \, j_! \, \,}}
\end{xy}.
\end{align}
By \cite[Theorem 4]{NicolasSaorin09} (which needs the flatness assumption) and the first paragraph after Lemma 4 of~loc.~cit., there exists a dg algebra $B'$ and a morphism of dg-algebras
 $f':A\rightarrow B'$  such that there is a recollement  \eqref{E:RecollmentByIdempotent} and the adjoint triple $(i^*,i_*=i_!,i^!)$
is given as in (a) with $B$ replaced by $B'$. We claim that $H^i(B')=0$ for $i>0$ and that $H^0(f')$ induces an isomorphism of algebras
$A/AeA\cong H^0(B')$. Then taking $B=\sigma^{\leq 0}B'$
and $f=\sigma^{\leq 0}f'$ finishes the proof for (a) (b) and (c).

In order to prove the claim, we take the distinguished triangle associated to $A$.
\begin{eqnarray}\label{e:canonical-triangle-projective}
\begin{array}{cc}
\begin{tikzpicture}[description/.style={fill=white,inner sep=2pt}]
    \matrix (n) [matrix of math nodes, row sep=1em,
                 column sep=2.5em, text height=1.5ex, text depth=0.25ex,
                 inner sep=0pt, nodes={inner xsep=0.3333em, inner
ysep=0.3333em}]
    {  
       Ae\lten_{eAe}eA & A & B' & Ae\lten_{eAe}eA[1] \\
       j_!j^!(A) && i_*i^*(A) \\
    };
\draw[->] (n-1-1.east) -- node[scale=0.75, yshift=2.5mm] [midway] {$\varphi$} (n-1-2.west);    
\draw[->] (n-1-2.east) -- node[scale=0.75, yshift=2.5mm] [midway] {$f'$} (n-1-3.west);  
\draw[->] (n-1-3.east) -- (n-1-4.west); 
\draw[-] ($(n-1-1.south)+(-0.5mm,0)$) -- ($(n-2-1.north)+(-0.5mm,0)$);
\draw[-] ($(n-1-1.south)+(.5mm,0)$) -- ($(n-2-1.north)+(0.5mm,0)$);
\draw[-] ($(n-1-3.south)+(-1mm,1mm)$) -- ($(n-2-3.north)+(-1mm,0)$);
\draw[-] ($(n-1-3.south)+(0mm,1mm)$) -- ($(n-2-3.north)+(0mm,0)$);
\end{tikzpicture}
\end{array}\end{eqnarray}
By applying $H^0$ to the triangle (\ref{e:canonical-triangle-projective}) we obtain a long exact cohomology sequence
\[
\begin{tikzpicture}[description/.style={fill=white,inner sep=2pt}]
    \matrix (n) [matrix of math nodes, row sep=1em,
                 column sep=2.25em, text height=1.5ex, text depth=0.25ex,
                 inner sep=0pt, nodes={inner xsep=0.3333em, inner
ysep=0.3333em}]
    {  
       \, & H^i(Ae\lten_{eAe}eA) & H^i(A) & H^i(B') & H^{i+1}(Ae\lten_{eAe}eA) & \,  \\
    };
\draw[dotted, -] (n-1-1.east) -- (n-1-2.west);    
\draw[->] (n-1-2.east) -- node[scale=0.75, yshift=3mm] [midway] {$H^i(\varphi)$} (n-1-3.west);  
\draw[->] (n-1-3.east) -- node[scale=0.75, yshift=3mm] [midway] {$H^i(f')$} (n-1-4.west); 
\draw[->] (n-1-4.east) -- (n-1-5.west); 
\draw[dotted, -] (n-1-5.east) -- (n-1-6.west); 
\end{tikzpicture}
\]
If $i>0$, both $H^i(A)$ and $H^{i+1}(Ae\lten_{eAe}eA)$ are trivial, and hence $H^i(B')$ is trivial.
If $i=0$, then $H^0(B')\cong H^0(A)/\im(H^0(\varphi))$.
But $H^0(Ae\lten_{eAe}eA)\cong Ae\ten_{eAe}eA$ and the image of $H^0(\varphi)$ is precisely $AeA$.
Therefore $H^0(f'):A\rightarrow H^0(B')$ induces an isomorphism $H^0(B')\cong A/AeA$, which is clearly a homomorphism of algebras.
\end{proof}

\begin{remark}\label{r:B-augmented}
Assume that $k$ is a field and $A/AeA$ is finite-dimensional over $k$. Let $K=(A/AeA)/\rad(A/AeA)$. Assume that $K$ is a direct product of finitely many (say, $r$) copies of $k$. We claim that $B$ has the structure of  an augmented dg algebra over $K$ up to quasi-equivalence in the sense of \cite[Section 7]{Keller94}. Since $\Hom_{\cd(B)}(B,B)=H^0(B)\cong A/AeA$, it follows that $B$ is isomorphic in $\cd(B)$ to the direct sum of $r$ dg $B$-modules, say, $Y_1,\ldots,Y_r$, which are indecomposable and pairwise non-isomorphic in $\cd(B)$. We assume that $Y$ is $\ch$-projective over $B$ and let $C=\sigma^{\leq 0}\cEnd_B(Y)$. Then $Y$ is a quasi-equivalence from $C$ to $B$. It is straightforward to check that $C$ is a non-positive dg algebra over $K=k\{id_{Y_1}\}\times\cdots\times k\{id_{Y_r}\}$ and is augmented. The quasi-equivalence $Y$ yields a derived equivalence $?\lten_C Y: \cd(C)\to \cd(B)$, which takes $C$ to $Y\cong B$, and which restricts to triangle equivalences $\per(C)\to\per(B)$ and $\cd_{fd}(C)\to\cd_{fd}(B)$.
\end{remark}

\begin{corollary}\label{c:restriction-and-induction} Keep the assumptions and notations as in Proposition~\ref{p:recollement-from-projective-general-case}.
\begin{itemize}
 \item[(a)]
The functor
$i^*$ induces an equivalence of triangulated categories 
\begin{align} \label{E:Neeman} \left(K^b(\proj-A)/\thick(eA)\right)^\omega\stackrel{\sim}{\longrightarrow}\per(B), \end{align}
where $(-)^\omega$ denotes the idempotent completion (see \cite{BalmerSchlichting01}).

 \item[(b)] Let $k$ be a field. Let $\cd_{fd,A/AeA}(A)$ be the full subcategory of $\cd_{fd}(A)$ consisting of
complexes with cohomologies supported on $A/AeA$. The functor $i_*$
induces a triangle equivalence
$\cd_{fd}(B)\stackrel{\sim}{\longrightarrow}\cd_{fd,A/AeA}(A)$.
Moreover, the latter category coincides with
$\thick_{\cd(A)}(\fdmod-A/AeA)$.
\end{itemize}
\end{corollary}
\begin{proof}
(a) Since $j_{!}(eAe)=eAe \lten_{eAe} eA \cong eA$, $eAe$ generates $\cd(eAe)$ and $j_{!}$ commutes with direct sums, we obtain $\im j_{!}=\Tria(eA)$. Hence, 
$i^*$ induces a triangle equivalence \begin{align} \label{E:BigQuot} \cd(A)/\Tria(eA)\cong\cd(B). \end{align} As a projective $A$-module $eA$ is compact in $\cd(A)$. By definition, $\Tria(eA)$ is the smallest localizing subcategory containing $eA$. Since $\cd(A)$ is compactly generated, Neeman's interpretation (and generalization) \cite[Theorem 2.1]{Neeman92a} of Thomason \& Trobaugh's and Yao's Localization Theorems shows that  restricting \eqref{E:BigQuot} to the subcategories of compact objects yields a triangle equivalence $K^b(\proj-A)/\thick(eA) \ra \per(B)$ up to direct summands. Hence, the equivalence (\ref{E:Neeman}) follows. 

(b) By construction of the dg algebra $B$ in Proposition \ref{p:recollement-from-projective-general-case},  $i_*$ induces a triangulated equivalence between $\cd(B)$ and $\cd_{A/AeA}(A)$, the full subcategory of $\cd(A)$ consisting of
complexes of $A$-modules which have cohomologies supported on $A/AeA$.
Moreover, $i_*$ restricts to a triangle equivalence between $\cd_{fd}(B)$
and $i_*(\cd_{fd}(B))$. The latter category is contained in
$\cd_{fd}(A)$ because $i_*$ is the restriction along the
homomorphism $f:A\rightarrow B$, and hence is contained in
$\cd_{fd}(A)\cap\cd_{A/AeA}(A)=\cd_{fd,A/AeA}(A)$, which
in turn is contained in $\thick_{\cd(A)}(\fdmod-A/AeA)$. By
Proposition~\ref{p:standard-t-str} (b), $\fdmod-H^0(B)$ generates
$\cd_{fd}(B)$. But $i_*$ induces an equivalence from $\fdmod-H^0(B)$
to $\fdmod-A/AeA$. Therefore
$\thick_{\cd(A)}(\fdmod-A/AeA)=i_*(\cd_{fd}(B))$, and hence
$\thick_{\cd(A)}(\fdmod-A/AeA)=i_*(\cd_{fd}(B))=\cd_{fd,A/AeA}(A)$.
We are done.
\end{proof}

The next result is used in the proof of Theorem \ref{t:main-thm-2} and is also interesting in itself.

\begin{corollary}\label{C:Hom-finite}
We keep the notations and assumptions of Proposition \ref{p:recollement-from-projective-general-case}. We assume additionally that $k$ is a field, $A/AeA$ is finite-dimensional over $k$ and every simple $A/AeA$-module has finite projective dimension when viewed as an $A$-module. Then the following statements hold
\begin{itemize}
\item[$(i)$] $H^i(B)$ is finite-dimensional for all $i$. In particular, $\per(B)$ is Hom-finite.
\item[$(ii)$] $K^b(\proj-A)/\thick(eA)$ is Hom-finite.
\end{itemize}
\end{corollary}
\begin{proof}
Corollary \ref{c:restriction-and-induction}(a) shows that $(i)$ implies $(ii)$. Statement $(i)$ follows from Proposition \ref{p:hom-finiteness-of-per}. To apply this proposition, it suffices to show $\cd_{fd}(B) \subseteq \per(B)$ since the other conditions hold by Proposition \ref{p:recollement-from-projective-general-case} and our assumptions.  Let $M \in \cd_{fd}(B)$. Using Corollary  \ref{c:restriction-and-induction}(b) and our assumption that all finite-dimensional $A/AeA$ modules have finite projective dimension over $A$, we obtain $i_*(M) \in K^b(\proj-A)$. Remark \ref{r:identity-units-and-counits} shows that $M\cong i^*i_*(M)$. Since $i^*$ respects compact objects \cite[Theorem 2.1]{Neeman92a}, $i^*(K^b(\proj-A)) \subseteq \per(B)$. This completes the proof.\end{proof}

\begin{remark}\label{r:CompInterprOfRelSingCat}
The triangle equivalences \eqref{E:Neeman} and $\cd(B) \cong \cd_{A/AeA}(A)$ show that 
\begin{align}
\left(\cd_{A/AeA}(A)\right)^{c} \cong \left(K^b(\proj-A)/\thick(eA)\right)^\omega.
\end{align}
In particular, if $A$ is a non-commutative resolution of a complete Gorenstein singularity $R$, then the relative singularity category $\Delta_{R}(A) \cong \cd^b(\mod-A)/\thick(eA)$ is idempotent complete by \cite[Section 3]{BurbanKalck11}. Hence, there is a triangle equivalence
\begin{align}
\left(\cd_{A/AeA}(A)\right)^{c} \cong \Delta_{R}(A).
\end{align}
\end{remark}

\section{Quotient functors associated with idempotents}\label{S:Tale}

\begin{definition}
A triangulated functor $\F\colon \cc \ra \cd$ is called \emph{triangulated quotient functor} if the induced functor $\ul{\F}\colon \cc/\ker \F \ra \cd$ is an equivalence of categories.
\end{definition}

The following well-known lemma may be seen as an triangulated category analogue of the third isomorphism theorem for groups.

\begin{lemma}\label{L:Verdier}
Let $\F \colon \cc \ra \cd$ be a triangulated quotient functor with kernel $\ck$. Let $\cu \subseteq \cc$ be a full triangulated subcategory, let $q\colon\cc \ra \cc/\cu$ be the quotient functor and $\cv=\thick(\F(\cu))$. Then $\ol{\F}\colon \cc/\cu \to \cd/\cv$ is a quotient functor with kernel $\thick\left(q(\ck)\right)$. In particular, $\ol{\F}$ induces an equivalence of triangulated categories. 
$$\frac{(\cc/\cu)}{\thick(q(\ck))} \longrightarrow \frac{\cd}{\cv}.$$
\end{lemma}
\begin{proof}
$\F$ induces a triangle functor $\ol{\F}\colon \cc/\cu \ra \cd/\cv$. We have $\thick(q(\ck)) \subseteq \ker(\ol{\F})$. We have to show that $\ol{\F}$ is universal with this property,  see e.g.~\cite[Theorem 2.1.8 \& Remark 2.1.10]{Neeman99}. Let $\G\colon \cc/\cu \ra \ct$ be a triangle functor with $\thick(q(\ck)) \subseteq \ker(\G)$. We explain the following commutative diagram.
$$
\begin{xy}
\SelectTips{cm}{}
\xymatrix@R=13pt@C=40pt
{
\ck \ar[dd]_{\displaystyle q} \ar[rr] && \cc \ar[rr]^{\displaystyle \F} \ar[dd]_{\displaystyle q} && \cd \ar[dd]^{\displaystyle q'} \ar@{-->}[dl]_(.64){\displaystyle \I_{1}} \\
&&& \ct \\
\thick\left(q(\ck)\right) \ar[rr] && \cc/\cu \ar[ru]^{ \displaystyle \G} \ar[rr]^{\displaystyle \ol{\F}} && \cd/\cv. \ar@{-->}[lu]_{\displaystyle \I_{2}}
}
\end{xy}
$$
$\I_{1}$ exists by the universal property of $\F$ and $\I_{2}$ exists by the universal property of $q'$. Since  $\I_{2}\circ \ol{\F}\circ q=\I_{1}\circ\F=\G \circ q$ the universal property of $q$ implies $\I_{2}\circ\ol{\F}=\G$.

To show uniqueness of $\I_{2}$ let $\mathbb{H}\colon \cd/\cv \ra \ct$ be a triangle functor such that $\mathbb{H}\circ\ol{\F}=\G$. Then $\mathbb{H} \circ q' \circ \F = \G \circ q$ and the universal property of $\F$ imply $\mathbb{H} \circ q'=\I_{1}$. Since $\mathbb{H} \circ q' =\I_{1}= \I_{2} \circ q'$ the universal property of $q'$ yields $\I_{2}=\mathbb{H}$.
\end{proof} 

\noindent
The case $e=f$ of the following proposition will be the most important in the sequel.

\begin{proposition}\label{P:TwoIdempotents}
Let $A$ be a right Noetherian ring and let $e, f \in A$ be idempotents. The exact functor $\F=\Hom_{A}(eA, -)$ induces an equivalence of triangulated categories
\begin{align}\label{E:twoidempotents}
\frac{\cd^b(\mod-A)/\thick(fA)}{\thick(q(\mod-A/AeA))} \longrightarrow \frac{\cd^b(\mod-eAe)}{\thick(fAe)}.
\end{align}
\end{proposition}
\begin{proof}
On the abelian level $\F$ induces a well-known equivalence
\begin{align}\label{E:OneIdempAbelian}
\ul{\F}\colon \frac{\mod-A}{\mod-A/AeA} \longrightarrow \mod-eAe,
\end{align}
which may be deduced from an appropriate version of \cite[Proposition III.5]{Gabriel62} in conjunction with classical Morita theory (see e.g.~\cite[Theorem 8.4.4.]{DrozdKirichenko}). Using a compatibility result, which relates abelian quotients with triangulated quotients of derived categories {\cite[Theorem 3.2.]{Miyachi}}, the equivalence (\ref{E:OneIdempAbelian}) shows that $\F$ induces a triangulated \emph{quotient} functor
$\F\colon \cd^b(\mod-A) \ra \cd^b(\mod-eAe)$. An application of Lemma \ref{L:Verdier} to $\F$ and $\thick(fA)$ completes the proof.
\end{proof}

\begin{remark}
Proposition \ref{P:TwoIdempotents} contains Chen's \cite[Theorem 3.1]{Chen10} as a special case. Namely, if we set $f=1$ and assume that $\prdim_{eAe}(Ae)< \infty$ holds, (\ref{E:twoidempotents}) yields a triangle equivalence
$\cd_{sg}(A)/\thick(q(\mod-A/AeA)) \ra \cd_{sg}(eAe)$. If moreover every finitely generated $A/AeA$-module has finite projective dimension over $A$ (i.e.~the idempotent $e$ is \emph{singularly-complete} in the terminology of loc.~cit.), we get an equivalence $\cd_{sg}(A) \ra \cd_{sg}(eAe)$ of singularity categories \cite[Corollary 3.3]{Chen10}.
\end{remark}
\begin{remark}
Lemma \ref{L:Verdier} has another application. Let $\mathbb{X}=(X, \ca)$ be a (non-commutative) ringed space as studied in Burban \& Kalck \cite{BurbanKalck11}, see also Subsection \ref{ss:Global}. Let $j\colon U \ra X$ be an open immersion. The restriction functor $j^*\colon \mathsf{Perf}(X) \ra \mathsf{Perf}(U)$ is essentially surjective by \cite[Lemma 5.5.1]{ThomasonTrobaugh90}. Moreover, $j^*\colon \cd^b(\Coh(\ca)) \ra \cd^b\!\left(\Coh\left(\left.\ca\right|_{U}\right)\right)$ is a triangulated quotient functor. Hence, Lemma \ref{L:Verdier} yields a triangle equivalence \begin{align}
\frac{\cd^b(\Coh(\ca))/\mathsf{Perf}(X)}{\thick(q(\Coh_{(X \setminus U)}(\ca)))} \longrightarrow \frac{\cd^b\!\left(\Coh\left(\left.\ca\right|_{U}\right)\right)}{\mathsf{Perf}(U)}
\end{align}
In combination with \cite[Proposition 2.6]{BurbanKalck11} this yields a proof of \cite[Theorem 2.7]{BurbanKalck11}. 
This is analogous to the commutative case treated in \cite[Proposition 1.2]{Chen10}.
\end{remark}

\section{The fractional Calabi--Yau property}\label{ss:fract-cy}

Let $\ce$ be an idempotent complete Frobenius $k$-category. 

\begin{definition}\label{d:almost-split-sequence}
Let $\cc$ be a full additive subcategory of $\ce$. We say that $\cc$ \emph{has $d$-almost split sequences} if $\cc$ is Krull--Schmidt and for any indecomposable object $X$ of $\cc$, which is not projective in $\ce$ (respectively, indecomposable object $Y$ of $\cc$, which is not injective in $\ce$) there is an exact sequence (called a \emph{$d$-almost split sequence}, see~\cite[Section 3.1]{Iyama07a})
\begin{align}
 0\longrightarrow Y \stackrel{f_d}{\longrightarrow} C_{d-1} \stackrel{f_{d-1}}{\longrightarrow} \ldots \longrightarrow C_0 \stackrel{f_0}{\longrightarrow} X \longrightarrow 0\nonumber
\end{align}
with terms in $\cc$ and $f_d,\ldots,f_0$ belong to the Jacobson radical $J_\cc$ of $\cc$ such that the following two sequences of functors are exact
\begin{align} 
0 \longrightarrow (?,Y) \stackrel{\cdot f_{d}}{\longrightarrow} (?,C_{d-1}) \longrightarrow \ldots \longrightarrow (?,C_0) \stackrel{\cdot f_{0}}{\longrightarrow} J_\cc(?,X) \longrightarrow & 0,\nonumber
\end{align}
\begin{align} 
0 \longrightarrow (X,?) \stackrel{f_{0}\cdot}{\longrightarrow} (C_{0},?) \longrightarrow \ldots \longrightarrow (C_{d-1},?) \stackrel{f_{d}\cdot}{\longrightarrow} J_\cc(Y,?) \longrightarrow & 0,\nonumber
\end{align}
where  $(X, Y)=\Hom_{\cc}(X, Y)$. As in \cite[Proposition 3.1.1]{Iyama07a} one can show that $X$ is indecomposable if and only if $Y$ is indecomposable and that $X$ and $Y$ mutually determine each other. In particular, one gets well-defined maps on isomorphism classes of indecomposables by setting $\tau_d(X)=Y$ (respectively, $\tau_d^{-1}(Y)=X$).
\end{definition}

We assume that there exists $P \in \ce$ such that $\proj\ce=\add(P)$. Let $\cc=\add(F)$ with $F=P \oplus F'$. We assume that $\cc$ has $d$-almost split sequences. In particular, $\cc$ is Krull-Schmidt and therefore we may assume that $F \cong P \oplus F_1\oplus\ldots\oplus F_r$ such that $F_1,\ldots,F_r$ are pairwise non-isomorphic indecomposable objects, which are not projective in $\ce$.
Let $A=\End_{\ce}(F)$, $R=\End_{\ce}(P)$ and $e=id_P\in A$. Then $A/AeA$ is the stable endomorphism algebra of $F$. 

Recall from Corollary~\ref{c:restriction-and-induction} (b) that $\cd_{fd,A/AeA}(A)$ 
denotes the full subcategory of $\cd_{fd}(A)$ consisting of complexes of $A$-modules which have
cohomologies supported on $A/AeA$ and that $\cd_{fd,A/AeA}(A)$ is generated by $\fdmod-A/AeA$. In particular, if $A/AeA$ is finite-dimensional over $k$, then
$\cd_{fd,A/AeA}(A)$ is
generated by the simple $A/AeA$-modules $S_1,\ldots,S_r$, corresponding to $F_1,\ldots,F_r$, respectively.

\begin{theorem}\label{t:fractionally-cy-property}
In the notation of the paragraph above, if $\add(F)$ has $d$-almost split
sequences and $A/AeA$ is finite-dimensional over $k$, then the following statements hold.
\begin{itemize}
\item[(a)] Any finite-dimensional $A/AeA$-module has finite projective dimension over~$A$. 
\item[(b)] The triangulated category $\cd_{fd,A/AeA}(A)$ admits a Serre functor $\nu$. 
\item[(c)] For $i=1,\ldots,r$, the simple $A/AeA$-module $S_i$ is fractionally $\frac{(d+1)n_i}{n_i}$--CY, where $n_i$ is the smallest positive integer such that $\tau_d^{n_i}(F_i)\cong F_i$ holds. 
\item[(d)] There exists a permutation $\pi$ on the isomorphism classes of simple $A/AeA$-modules such that
$D\Ext_A^l(S,S')\cong \Ext^{d+1-l}_A(S',\pi(S)),$ holds for all  $l \in \Z.$
\end{itemize}
\end{theorem}

\begin{proof}
 For $i=1,\ldots,r$, let $e_i=\id_{F_i}$ and consider it as an element in $A$. Then $1_A=e+e_1+\ldots+e_r$. Let $S_i$ be the simple $A$-module
corresponding to $e_i$. By assumption there is an $d$-almost split sequence (see Definition~\ref{d:almost-split-sequence})
\begin{align}
\eta\colon 0\longrightarrow F_{\pi(i)} \longrightarrow C_{d-1} \longrightarrow \ldots \longrightarrow C_0 \longrightarrow F_i \longrightarrow 0
\end{align}
where $C_{d-1},\ldots,C_0\in\add(F)$ and $\pi$ is the permutation
on the set $\{1,\ldots,r\}$ induced by $F_{\pi(i)}=\tau_d(F_i)$.
Applying $\Hom_{\ce}(F, ?)$ to $\eta$ yields a $A$-projective resolution of $S_i$
\begin{align} \label{eq:proj-res}
0 \longrightarrow (F,F_{\pi(i)}) \longrightarrow (F,C_{d-1}) \longrightarrow \ldots \longrightarrow (F,C_0) \longrightarrow (F,F_i) \longrightarrow S_i \longrightarrow & 0.
\end{align}
In particular, this shows (a). Dually, we acquire an $A$--injective resolution of $S_{\pi(i)}$
\begin{align}\label{eq:inj-res}
 0 \rightarrow S_{\pi(i)} \rightarrow D(F_{\pi(i)},F)
\rightarrow D(C_{d-1},F) \rightarrow
\ldots \rightarrow D(C_0,F) \rightarrow D(F_i,F) \rightarrow
0,
\end{align}
 by applying $D\Hom_{\ce}(?, F)$ to $\eta$. Recall from Section~\ref{ss:nakayama-functor} that there is a
triangle functor $\nu \colon \per(A)\rightarrow \thick(DA)$. From \eqref{eq:proj-res} and \eqref{eq:inj-res}, we see that $\nu(S_i)=S_{\pi(i)}[d+1]$ and therefore the subcategory
$\cd_{fd,A/AeA}(A)=\thick(S_1,\ldots,S_r)\subseteq
\per(A)\cap\thick(DA)$ is invariant under $\nu$. It follows from 
formula (\ref{E:AR-formula}) that the restriction of
$\nu$ on $\cd_{fd,A/AeA}(A)$ is a right Serre functor and hence fully faithful~\cite[Corollary I.1.2]{ReitenVandenBergh02} --- note that $\cd_{fd,A/AeA}(A)\subseteq \per(A)$ is Hom-finite since $\Hom_{\cd(A)}(P, X)$ is finite-dimensional for all $P \in \per(A)$ and $X \in \cd_{fd,A/AeA}(A)$.  As shown above, $\nu$ takes a set of generators of $\cd_{fd,A/AeA}(A)$ to itself up to shift. It follows that 
$\nu$ restricts to an auto-equivalence of
$\cd_{fd,A/AeA}(A)$. Indeed, the essential image of a fully faithful triangle functor is a triangulated subcategory. Since the essential image of $\nu$ contains a set of generators for $\cd_{fd,A/AeA}(A)$, we are done. In particular, $\nu$ is a Serre functor of $\cd_{fd,A/AeA}(A)$. 
Moreover, if
$n$ denotes the number of elements in the $\pi$-orbit of $i$, then
$\nu^n(S_i)\cong S_i[(d+1)n]$, i.e.~ $S_i$ is fractionally Calabi--Yau
of Calabi--Yau dimension $\frac{(d+1)n}{n}$. Finally, we have a chain of isomorphisms
\begin{align*}
D\Ext_{A}^l(S_i,S_j)&\cong D\Hom_{A}(S_i,S_j[l])
\cong \Hom_{A}(S_j,\nu(S_i)[-l])\\
&\cong \Hom_{A}(S_j,S_{\pi(i)}[d+1-l])
\cong \Ext_{A}^{d+1-l}(S_j,S_{\pi(i)}),
\end{align*}
where $i,j=1,\ldots,r$
and $l$ denotes an integer. This proves part (d).
\end{proof}

\section{DG-Auslander algebras and relative singularity categories}\label{s:class-to-relative}

In this section, we show that the relative Auslander singularity categories $\Delta_\ce(\Aus(\ce))$ of certain representation-finite Frobenius categories $\ce$ are triangle equivalent to perfect derived categories $\per(B)$ for dg algebras $B$ which are determined by the stable categories $\ul{\ce}$. In Section \ref{S:MCM-over-Gor}, we apply this general result to the Frobenius categories of maximal Cohen--Macaulay modules over Gorenstein rings and module categories over finite-dimensional selfinjective algebras.

\begin{setup}\label{S:FM} Let $\ce$ be a Frobenius $k$-category over an algebraically closed field $k$ satisfying the following conditions:
\begin{itemize}
\item[(FM1)] $\ce$ has $1$-almost split sequences, also known as Auslander-Reiten sequences (see e.g.~Definition \ref{d:almost-split-sequence})
 --- in particular, $\ce$ is Krull-Schmidt and therefore idempotent complete, see e.g.~\cite[Lemma 3.2.2]{Krause};
\item[(FM2)] $\ce$ has only finitely many isomorphism classes of indecomposable objects. We denote their representatives by $N_{1}, \ldots, N_{s}$; 
\item[(FM3)] the \emph{Auslander algebra} $A=\End_{\ce}(\bigoplus_{i=1}^s N_{i})$ of $\ce$ is right Noetherian;
\item[(FM4)] the stable category $\ul{\ce}=\ce/\!\proj \ce$ is Hom-finite and idempotent complete.
\end{itemize}
By (FM2), the category $\proj\ce$ of projective objects of $\ce$ has an additive generator $P$. Let $e \in A$ be the idempotent endomorphism corresponding to $\id_{P}$. We define the \emph{relative Auslander singularity category} of $\ce$ as the following triangle quotient 
\begin{align}\label{def:relative-auslander-singularity-category} \Delta_{\ce}(A)=\frac{K^b(\proj-A)}{\thick(eA)}. \end{align}
\end{setup}


\subsection{Independence of the Frobenius model}
Let $\ce$ be a Frobenius category as in Setup \ref{S:FM} above. Then $\ct=\ul{\ce}$ is an idempotent complete Hom-finite triangulated category with only finitely many isomorphism classes of indecomposable objects,
say $M_1,\ldots, M_r$ (we refer to Happel \cite{Happel88} for the triangulated structure on $\ct$). By \cite[Theorem 1.1]{Amiot07a},
$\ct$ has a Serre functor and thus has Auslander--Reiten triangles \cite[Theorem I.2.4]{ReitenVandenBergh02}.
Let $\tau$ be the Auslander--Reiten translation. By abuse of notation, $\tau$ also denotes the
induced permutation on $\{1,\ldots,r\}$ defined by $M_{\tau(i)}=\tau M_i$.

The quiver of the \emph{Auslander algebra} $\Lambda(\ct)=\End_{\ct}(\bigoplus_{i=1}^r M_i)$ of $\ct$ 
is the quiver of irreducible maps $\Gamma$ of $\ct$, in which we identify the vertex $i$ with the object $M_i$. We need the conditions (A1)--(A3) of the following lemma in our definition of dg Auslander algebra below. 

\begin{lemma}\label{l:existence-of-nice-minimal-relations} There exists a sequence of elements $\gamma=\{\gamma_1,\ldots,\gamma_r\}$ in $\widehat{k\Gamma}$ which satisfies the following conditions:
\begin{itemize}
 \item[(A1)]  for each vertex $i$ the element $\gamma_i$ is a (possibly infinite) combination of paths of $\Gamma$
from $i$ to $\tau^{-1}i$ of length $\geq 2$;
\item[(A2)] if $\Gamma$ has at least one arrow starting in $i$, then $\gamma_i$ is non-zero;
\item[(A3)] the non-zero
$\gamma_i$'s form a set of minimal relations for $\Lambda(\ct)$ (see Section~\ref{ss:minimal-relation}).
\end{itemize}
\end{lemma}
\begin{proof}
This forms a part of the proof of Theorem \ref{t:main-thm-2}.
\end{proof}

\noindent
We introduce the notion of a dg Auslander algebra, which plays a key role in the sequel.

\begin{definition}\label{d:dg-aus-alg} Let $\ct = \ul{\ce}$ be a triangulated category as above and let $\gamma=\{\gamma_1,\ldots,\gamma_r\}$ be a sequence satisfying the conditions (A1)--(A3) as in Lemma~\ref{l:existence-of-nice-minimal-relations}. 
The \emph{dg Auslander algebra} $\Lambda_{dg}(\ct,\gamma)$ of $\ct$ with respect to $\gamma$
is the dg algebra $(\widehat{kQ},d)$, where $Q$ is a graded quiver and 
$d\colon \widehat{kQ}\rightarrow\widehat{kQ}$ is a map such that
\begin{itemize}
\item[(dgA1)] $Q$ is concentrated in degrees $0$ and $-1$;
\item[(dgA2)] the degree $0$ part of $Q$ is the same as the quiver of irreducible maps $\Gamma$ of
$\ct$;
\item[(dgA3)]for each vertex $i$, there is precisely one arrow $\begin{xy} \SelectTips{cm}{} \xymatrix{\rho_i\colon i \ar@{-->}[r] & \tau^{-1}(i)} \end{xy}$ of degree
$-1$;
\item[(dgA4)] $d$ is the unique continuous $k$-linear map
on $\widehat{kQ}$ of degree $1$ satisfying the graded Leibniz rule
and taking $\rho_i$ ($i\in Q_0$) to the relation $\gamma_i$.
\end{itemize}
\end{definition}

\noindent It turns out that the dg Auslander algebra does not depend on the choice of $\gamma$:

\begin{proposition}\label{p:dg-aus-alg} Let $\ct=\ul{\ce}$ be as above. Let $\gamma=\{\gamma_1,\ldots,\gamma_r\}$ and $\gamma'=\{\gamma'_1,\ldots,\gamma'_r\}$ be sequences of elements of  $\widehat{k\Gamma}$ satisfying the conditions (A1)--(A3). Then the dg Auslander algebras $\Lambda_{dg}(\ct,\gamma)$ and $\Lambda_{dg}(\ct,\gamma')$ are isomorphic as dg algebras.
\end{proposition}
\begin{proof}
See Subsection \ref{sss:proofprop} below.
\end{proof}

Henceforth, we denote by $\Lambda_{dg}(\ct)$ the dg Auslander algebra of $\ct$ with respect to any sequence $\gamma$ satisfying (A1)--(A3). 

\begin{theorem}\label{t:main-thm-2}
Let $\ce$ be a Frobenius category as in Setup \ref{S:FM}. Denote by $\ct=\ul{\ce}$ the stable category and let $A=\Aus(\ce)$ be the Auslander algebra of $\ce$. Then the following statements hold 
\begin{itemize}
 \item[(a)] $\Delta_{\ce}(A)$ is triangle equivalent to $\per(\Lambda_{dg}(\ct))$ (up to direct summands);
 \item[(b)] $\Delta_{\ce}(A)$ is Hom-finite.
\end{itemize}
\end{theorem}
\begin{remark}
If $\Delta_{\ce}(A)$ is idempotent complete, then we can omit the supplement `up to direct summands' in the statement above. In particular, this holds in the case $\ce=\MCM(R)$, where $(R, \mathfrak{m})$ is a local complete Gorenstein $(R/\mathfrak{m})$-algebra \cite{BurbanKalck11}.
\end{remark}

\begin{proof}
Let $P$ be the direct sum of indecomposable projective objects of $\ce$ (each isomorphism class occurring with multiplicity one) and $e \in A$ be the idempotent endomorphism corresponding to the projection on $P$.
By
Corollary~\ref{c:restriction-and-induction} (a) 
and Remark~\ref{r:B-augmented},
there is a non-positive dg $k$-algebra $B$ which is augmented over $K=H^0(B)/\rad H^0(B)$ with $H^0(B)\cong A/AeA$, such that the idempotent completion of the relative singularity category $\left(\Delta_{\ce}(A)\right)^\omega:=\left(K^b(\proj-A)/\thick(eA)\right)^\omega$ is triangle equivalent to $\per(B)$.

The stable Auslander algebra $A/AeA$ is finite-dimensional,
since $\ct$ is Hom-finite by assumption. Theorem \ref{t:fractionally-cy-property}(a) shows that the 
simple $A/AeA$-modules have finite projective dimension over $A$. Corollary \ref{C:Hom-finite} implies statement (b) and in particular, that $H^i(B)$ is finite-dimensional for all $i \in \mathbb{Z}$. This is needed in our Koszul duality argument below.  

Namely, by Theorem~\ref{t:koszul-duality-for-non-positive-dg-alg} (b), $B$ is quasi-isomorphic to $E(B^*)$ as a dg algebra, where $B^*$ is the $A_\infty$-Koszul dual of $B$. Let $S=(A/AeA)/\rad(A/AeA)$. Then
$S\cong \bigoplus_{i=1}^r S_i$ as a dg $B$-module, where $S_1,\ldots,S_r$ is a complete set of pairwise non-isomorphic simple $A/AeA$-modules.
$B^*$ is $A_\infty$-quasi-isomorphic to the dg algebra $\RHom_B(S,S)$ by Theorem~\ref{t:koszul-duality-for-non-positive-dg-alg} (a). Because there is a dg algebra homomorphism $A\rightarrow B$ inducing an embedding $\cd(B)\rightarrow \cd(A)$ (see Proposition~\ref{p:recollement-from-projective-general-case}), the dg algebra $\RHom_B(S,S)$ is quasi-isomorphic to $\RHom_A(S,S)$. Thus as a graded algebra $B^*$ is
isomorphic to $\Ext_A^*(S,S)$. It follows from Theorem~\ref{t:fractionally-cy-property} that $\Ext_A^*(S,S)$
is concentrated in degrees $0$, $1$ and $2$. Clearly
$\Ext_A^0(S_i,S_j)=0$ unless $i=j$ in which case it is $k$. Analysing the proof of Theorem~\ref{t:fractionally-cy-property}
shows that in the current situation the permutation $\pi$ coincides with $\tau$. Therefore,
$\Ext_A^2(S_i,S_j)=0$ unless $j=\tau(i)$. Hence, $E(B^*)=(\widehat{kQ},d)$ for a graded quiver $Q$ and a continuous $k$-linear
differential $d$ of degree $1$, where the graded quiver $Q$ is concentrated in degree $0$ and $-1$,
and starting
from any vertex $i$ there is precisely one arrow $\rho_i$ of degree $-1$  whose target is $\tau^{-1}i$. This shows (dgA1) and (dgA3) in Definition \ref{d:dg-aus-alg}.

Let $Q^0$ denote the
degree $0$ part of $Q$. By the paragraph after Theorem \ref{t:koszul-double-dual}, we know that $Q^0$ is the quiver of $H^0(E(B^*))\cong H^0(B)\cong A/AeA=\Lambda(\ct)$. In other words, $Q^0$ is the quiver of irreducible maps $\Gamma$ of $\ct$, so $E(B^*)$ also satisfies (dgA2).

 Moreover, 
$\gamma=\{d(\rho_1),\ldots,d(\rho_r)\}$ is
a set of relations for $\Lambda(\ct)$, i.e. $\Lambda(\ct)=\widehat{kQ^0}/\overline{(d(\rho_i))}$. In other words, we have to show that $\im d^{-1}=\overline{(d(\rho_i))}$. The inclusion `$\subseteq$' follows from the continuity of $d$. To show the other inclusion, it suffices to prove that there exists $l>0$ such that $J^l \subseteq \im d^{-1}$, where $J \subseteq \widehat{kQ^0}$ is the ideal generated by the arrows. Let $q\colon \widehat{kQ^0} \to \widehat{kQ^0}/\im d^{-1}$ be the canonical projection. Then $q(J) \subseteq \rad \widehat{kQ^0}/\im d^{-1}$ is a nilpotent ideal, since $\widehat{kQ^0}/\im d^{-1} \cong H^0(E(B^*))\cong H^0(B)$ is finite-dimensional. This completes the argument. 

We claim that $\gamma$ is a sequence satisfying the conditions (A1)--(A3) of Lemma \ref{l:existence-of-nice-minimal-relations}. In particular, $E(B^*)$ satisfies (dgA4) with respect to $\gamma$. This completes the proof that $E(B^*)$ is isomorphic to the dg Auslander algebra $\Lambda_{dg}(\ct)$.

We prove the claim.
Let $J$ be the ideal of $\widehat{k\Gamma}$ generated by all arrows. Since $B^*$ is a minimal $A_\infty$-algebra, it follows that
$d(\rho_j)\in J^2$ holds for any $j=1,\ldots,r$, see Subsection~\ref{ss:dual-bar-construction}.
Since we already know that $ 
\begin{xy}
\SelectTips{cm}{}
\xymatrix{\rho_{i}\colon i \ar@{-->}[r] & \tau^{-1}(i)}
\end{xy}$ holds, $d(\rho_i)$ is a combination of paths of length $\geq 2$ from $i$ to $\tau^{-1}i$, for all $i=1, \ldots, r$. Hence, condition (A1) holds. In order to show (A2), we assume that $\Gamma$ has an arrow starting in $i$. 
Then there is an Auslander--Reiten triangle
\begin{align}\label{E:AR-triangle}
\begin{xy}
\SelectTips{cm}{}
\xymatrix{
M_i \ar[r]^{f_i} & X_i \ar[r]^{g_i} & M_{\tau^{-1}(i)} \ar[r] & M_i[1] 
}
\end{xy}
\end{align}
in $\ct$, where $f_i$ and $g_i$ are non-zero and irreducible. We may view $f_i$ and $g_i$ as elements 
of $\widehat{k\Gamma}$. The arrows\footnote{When we write `arrow', we also mean the corresponding irreducible map in $\ct$.} of $\Gamma$ which start in $i$ form a basis of the vector space of irreducible maps $\rad(M_i, X_i)/ \rad^2(M_i, X_i)$, see Happel \cite[Section 4.8]{Happel88}. In particular, $f_i$ may
be written as follows
\begin{align}\label{E:f}
f_i =\sum_{j=1}^m\lambda_j\alpha_j +r_i, 
\end{align}
where the $\alpha_j$ are the arrows starting in $i$, $r_i \in \rad^2(M_i, X_i)$ and the $\lambda_j \in k$ are not all zero. Similarly,
\begin{align}\label{E:g}
g_i =\sum_{j=1}^m\mu_j\beta_j +s_i, 
\end{align}
where the $\beta_j$ are the arrows ending in $\tau^{-1}(i)$, $s_i \in \rad^2(X_i,M_{\tau^{-1}(i)})$ and the $\mu_j \in k$ are not all zero. 

Define $m_i = g_i f_i$ in $\widehat{k\Gamma}$, which is a relation for $\Lambda(\ct)$, since (\ref{E:AR-triangle}) is a triangle. Therefore,
it is generated by $\{d(\rho_1),\ldots, d(\rho_r)\}$. In other words, 
there exists an index set $P$ and elements $c_{pj},c^{pj}\in\widehat{k\Gamma}$ ($(p,j)\in P
\times \{1,\ldots,r\}$) such that
\begin{eqnarray}\label{E:GenMeshRel}
 m_i &=& \sum_{j=1}^r\sum_{p\in P} c_{pj}d(\rho_j)c^{pj}.
\end{eqnarray}
 If $j\neq i$ and $c_{pj}d(\rho_j)c^{pj}\neq 0$,
then $c_{pj}d(\rho_j)c^{pj}$ is a combination of paths of length at least $4$, because $m_i$ is a combination of paths from $i$ to $\tau^{-1}i$,
while $d(\rho_j)$ is a combination of path of length $\geq 2$ from $j$ to $\tau^{-1}j$. Using \eqref{E:f} and \eqref{E:g}, we see that $m_i=g_i f_i$ has a non-zero length $2$ component. Thus \eqref{E:GenMeshRel} implies that $\sum_{p \in P} c_{pi}d(\rho_i)c^{pi}$
is non-zero and its length $2$ component equals that of $m_i$. 
In particular, $d(\rho_i)$ is non-zero and cannot be generated by 
$\{d(\rho_j)\}_{j\neq i}$. To summarise, $d(\rho_i)\neq 0$ if and only if $\Gamma$ has arrows starting in $i$ (A2), and the non-zero $d(\rho_i)$'s
form a set of minimal relations for $\Lambda(\ct)$ (A3).
\end{proof}

\begin{remark}
 Let $\ct$ be an idempotent complete Hom-finite algebraic triangulated category with only finitely 
many isomorphism classes of indecomposable objects. We say that $\ct$ is \emph{standard} if the Auslander algebra
$\Lambda(\ct)$ is given by the Auslander--Reiten quiver with mesh relations, see~\cite[Section 5]{Amiot07a}. Examples of
non-standard categories can be found in~\cite{Riedtmann83,Asashiba99}.

Assume that $\ct$ is standard and $\ct\cong
\underline{\ce}$ for some Frobenius category $\ce$ as in Setup \ref{S:FM}. Theorem~\ref{t:main-thm-2} shows that 
$\Delta_{\ce}(A)$ is determined by the Auslander--Reiten quiver of $\ct$ (up to direct summands).
\end{remark}

\subsection{Proof of Proposition \ref{p:dg-aus-alg} } \label{sss:proofprop}

By the assumptions (A1)--(A3), there exist $c_{i}\in k^{\times}$ ($i=1,\ldots,r$), an index set $P$ and $c_{pj},c^{pj}\in
\widehat{k\Gamma}$ ($(p,j)\in P\times \{1,\ldots,r\}$), such that for each pair $(p,j)$, at least one of $c_{pj}$ and $c^{pj}$ belongs to the ideal of $\widehat{k\Gamma}$ generated by all arrows,
and
\begin{eqnarray}\label{eq:gamma}
 \gamma'_i &=& c_{i}\gamma_i +\sum_{j=1}^r\sum_{p\in P}c_{pj}\gamma_jc^{pj}.
\end{eqnarray}
To see that $c_i \neq 0 $, using that the $\gamma'_j$ form a set of relations, we write
\begin{eqnarray}\label{eq:gammapr}
 \gamma_i &=& d_{i}\gamma'_i +\sum_{j=1}^r\sum_{p\in P}d_{pj}\gamma'_jd^{pj}
\end{eqnarray}
with $d_{pj}$ or $d^{pj}$ contained in the ideal of $\widehat{k\Gamma}$ generated by the arrows and $d_i \in k$. Using \eqref{eq:gamma} to substitute the $\gamma'_j$ in \eqref{eq:gammapr} we obtain
\begin{eqnarray}\label{eq:gammaf}
 \gamma_i &=& d_{i}c_i\gamma_i +\sum_{j=1}^r\sum_{p\in P}\tilde{d}_{pj}\gamma_j\tilde{d}^{pj}
\end{eqnarray}
This shows that $d_i c_i =1$, since $\tilde{d}_{pj}$ or $\tilde{d}^{pj}$ are contained in the arrow ideal.

We define a continuous graded $k$-algebra homomorphism $\varphi\colon\Lambda_{dg}(\ct,\gamma')\rightarrow \Lambda_{dg}(\ct,\gamma)$ as follows: it is the identity on the degree $0$ part and for arrows of degree $-1$ we set
\begin{eqnarray}\label{eq:rho}
 \varphi(\rho'_i) &=& c_{i}\rho_i +\sum_{j=1}^r\sum_{p\in P}c_{pj}\rho_j c^{pj}.
\end{eqnarray}
Since $\gamma_i=d(\rho_i)$ and $\gamma'_i=d(\rho'_i)$, it follows from (\ref{eq:gamma}) and (\ref{eq:rho}) that
$\varphi$ is a homomorphism of dg algebras. The equation
(\ref{eq:rho}), yields
\begin{eqnarray}\label{eq:rho'}
 \rho_i &=& c_i^{-1}\varphi(\rho'_i) -c_i^{-1}\sum_{j=1}^r\sum_{p\in P}c_{pj}\rho_jc^{pj}.
\end{eqnarray}
By iteratively substituting $c_j^{-1}\varphi(\rho'_j) -c_j^{-1}\sum_{k=1}^r\sum_{p\in P}c_{pk}\rho_jc^{pk}$ for $\rho_{j}$ on the right hand side of (\ref{eq:rho'}), we see
that there exists an index set $P'$ and elements $c'_{pj},c'^{pj}\in\widehat{k\Gamma}$ ($(p,j)\in P'\times \{1,\ldots,r\}$) such that
for each pair $(p,j)$ at least one of $c'_{pj}$ and $c'^{pj}$ belongs to the ideal of $\widehat{k\Gamma}$ generated by all arrows,
and the following equation holds
\begin{eqnarray}
\rho_i = c_i^{-1}\varphi(\rho'_i) -\sum_{j=1}^r\sum_{p\in P'}c'_{pj}\varphi(\rho'_j)c'^{pj}.
\end{eqnarray}
Define a continuous graded $k$-algebra homomorphism $\varphi'\colon\Lambda_{dg}(\ct,\gamma)\rightarrow \Lambda_{dg}(\ct,\gamma')$ as follows: $\phi'$ is the identity on the degree $0$ part and for arrows of degree $-1$ we set:
\begin{eqnarray}
\varphi'(\rho_i) &=& c_i^{-1}\rho'_i -\sum_{j=1}^r\sum_{p\in P'}c'_{pj}\rho'_j c'^{pj}.
\end{eqnarray}
It is clear that $\varphi\circ\varphi'=id$ holds.
Since $\varphi'$ and $\varphi$ have a similar form, the same argument as above shows that there exists a continuous graded $k$-algebra homomorphism $\varphi''\colon\Lambda_{dg}(\ct,\gamma')\rightarrow \Lambda_{dg}(\ct,\gamma)$ such that $\varphi'\circ\varphi''=id$ holds. Therefore we have $\varphi=\varphi''$. In particular, we see that $\varphi$ is an isomorphism.


\section[Classical vs. relative singularity categories for Gorenstein rings]{Classical vs. relative singularity categories for Gorenstein rings}\label{S:MCM-over-Gor}

Let $k$ be an algebraically closed field. Throughout this subsection $(R, \mathfrak{m})$ and $(R', \mathfrak{m}')$ denote commutative local complete Gorenstein $k$-algebras, such that their respective residue fields are isomorphic to $k$.

The results in this section actually hold in greater generality. Namely, we may (at least) replace $R$ and $R'$ respectively by Gorenstein $S$-orders in the sense of \cite[Section III.1]{Auslander76} or finite-dimensional selfinjective $k$-algebras. Here, $S=(S, \mathfrak{n})$ denotes a complete regular Noetherian $k$-algebra, with $k \cong S/\mathfrak{n}$. We decided to stay in the more restricted setup above to keep the exposition clear and concise. It is mostly a matter of heavier notation and not hard to work out the more general results. 
\subsection{Classical singularity categories}

Let $\MCM(R)$ be the category of maximal Cohen--Macaulay $R$-modules. Note that $\MCM(R)$ is a Frobenius category with $\proj \MCM(R) = \proj-R$ (see e.g.~ \cite{Buchweitz87}). Hence, $\ul{\MCM}(R)=\MCM(R)/\proj-R$ is a triangulated category \cite{Happel88}.

The following concrete examples of hypersurface rings are of particular interest: 
Let $R=\C\llbracket z_{0}, \ldots, z_{d}\rrbracket/(f)$, where $d \geq 1$ and $f$ is one of the following polynomials
\begin{itemize}
\item[$(A_{n})$] \quad  $z_0^2 + z_1^{n+1} + z^2_{2} + \ldots + z_{d}^2 \quad  \,\, \, \, \, ( n \geq 1 )$,
\item[$(D_{n})$] \quad $z_0^2z_1 + z_1^{n-1} + z^2_{2} + \ldots + z_{d}^2 \quad ( n \geq 4 )$,
\item[$(E_{6})$] \quad $z_0^3 + z_1^4 + z^2_{2} + \ldots + z_{d}^2$,
\item[$(E_{7})$] \quad $z_0^3 + z_0z_1^3 + z^2_{2} + \ldots + z_{d}^2$,
\item[$(E_{8})$] \quad $z_0^3 + z_1^5 + z^2_{2} + \ldots + z_{d}^2$.
\end{itemize}
Such a $\C$-algebra $R$ is called \emph{ADE--singularity} of dimension $d$. As hypersurface singularities they are known to be Gorenstein (see e.g.~ \cite{BrunsHerzog}). The following result is known as Kn\"orrer's Periodicity Theorem, see \cite{Knoerrer87}. It was the main motivation for Theorem \ref{t:main-thm-2} and Theorem \ref{t:Classical-versus-generalized-singularity-categories}.

\begin{theorem} \label{t:knoerrer}
Let $d \geq 1$ and $k$ an algebraically closed field such that $\mathrm{char}\, k \neq 2$. Let $S=k\llbracket z_{0}, \ldots, z_{d} \rrbracket$ and $f \in (z_{0}, \ldots, z_{d}) \setminus \{0\}$. Set $R=S/(f)$ and $R'=S\llbracket x, y \rrbracket/(f+xy)$. Then there is a triangle equivalence 
\begin{align}
\ul{\MCM}(R') \rightarrow \ul{\MCM}(R).
\end{align}
\end{theorem}

\begin{definition}
We say that $R$ is \emph{$\MCM$--finite} if there are only finitely many isomorphism classes of indecomposable maximal Cohen--Macaulay $R$--modules.
\end{definition}

\begin{remark}
Solberg \cite{Solberg89} showed that Theorem \ref{t:knoerrer} also holds in characteristic $2$ if $R$ is $\MCM$-finite.
\end{remark}

It follows from Theorem \ref{t:knoerrer} that $R$ is $\MCM$--finite if and only if $R'$ is $\MCM$-finite. The ADE--curve singularities are  $\MCM$-finite by work of Drozd \& Roiter \cite{DrozdRoiter} and Jacobinsky \cite{Jacobinsky}. Moreover, the ADE--surface singularities are $\MCM$-finite by work of Herzog  \cite{Herzog}. This has the following well-known consequence.

\begin{corollary} Let $R$ be an ADE--singularity as above. Then $R$ is $\MCM$--finite.
\end{corollary}

\begin{remark} 
If $k$ is an arbitrary algebraically closed field, then the ADE--polynomials listed above still describe $\MCM$--finite singularities. Yet there exist further $\MCM$--finite rings if $k$ has characteristic $2$, $3$ or $5$ (complete lists are contained in \cite{GreuelKroening90}).
\end{remark}

\subsection{Relative singularity categories}

Henceforth, let $F'$ be a finitely generated $R$-module and $F=R \oplus F'$. We call $A=\End_{R}(F)$ a \emph{partial resolution of $R$}. If $A$ has finite global dimension we say that $A$ is a \emph{resolution}. Denote by $e\in A$ the idempotent endomorphism corresponding to the identity morphism $\id_{R}$ of $R$. 

The situation is particularly nice if $R$ is $\MCM$--finite. Let $M_{0}=R, M_{1}, \ldots, M_{t}$ be representatives of 
the indecomposable objects of $\MCM(R)$. Their endomorphism algebra $\Aus(\MCM(R))=\End_{R}(\bigoplus_{i=0}^t M_{i})$ is called 
the \emph{Auslander algebra}. Auslander~{\cite[Theorem A.1]{Auslander84}} has shown that its global dimension is bounded above by the Krull dimension of $R$ (respectively by $2$ in Krull dimensions $0$ and $1$; for this case see also Auslander's treatment in \cite[Sections III.2 and III.3]{Auslander71}). Hence, $\Aus(\MCM(R))$ is a resolution of $R$.
 
 The next lemma motivates the definition  of the \emph{relative singularity categories}.

\begin{lemma}\label{L:embeds} 
There is a fully faithful triangle functor $K^b(\proj-R) \ra \cd^b(\mod-A)$.
\end{lemma}
\begin{proof}
By the definition of $A$, there exists an idempotent $e \in A$ such that $R \cong eAe$. Moreover, Proposition \ref{p:recollement-from-projective-general-case} (a) yields a fully faithful functor
\begin{align}
-\lten_{eAe}eA \colon \cd(eAe) \longrightarrow \cd(A).
\end{align}
This functor descends to an embedding between the categories of compact objects
\begin{align}
-\lten_{eAe}eA\colon K^b(\proj-eAe) \longrightarrow K^b(\proj-A) 
\end{align}
since the image of the generator $eAe$ is the projective $A$-module $eA$.
Composing with the canonical embedding $K^b(\proj-A) \subseteq \cd^b(\mod-A)$ completes the proof.
\end{proof}

\begin{remark} \label{R:FCM}
If $A$ is Cohen-Macaulay when considered as an $R$-module, then $F'$ is a Cohen-Macaulay $R$-module. Indeed, we have $Ae \cong R \oplus F'$ and $A \cong Ae \oplus A(1-e)$ as $R$-modules. So $F'$ is an $R$-direct summand of $A$ and therefore Cohen--Macaulay.  

In this situation, $A$ is a modification algebra in the sense of Iyama \& Wemyss \cite{IyamaWemyss14}. In particular, Van den Bergh's NCCRs [89] are of this form, \confer \cite[Lemma 2.23]{IyamaWemyss14}.
\end{remark}

\begin{remark} Assume that $F'$ is a maximal Cohen--Macaulay $R$-module --- by Remark \ref{R:FCM} this is satisfied if $A$ is in $\MCM(R)$, for example if $A$ is an NCCR in the sense of Van den Bergh \cite{NCCR}. Then the fully faithful functor $j_*=\RHom_{eAe}(Ae, ?)$ in Proposition \ref{p:recollement-from-projective-general-case} (a) also restricts to an embedding $ K^b(\proj-eAe) \to \cd^b(\mod-A)$ --- namely, representing $A$ as a matrix algebra one can check that $j_*(eAe)=eA$ using the vanishing condition in \eqref{mcm}. One can check that this embedding coincides with the embedding in Lemma \ref{L:embeds}.
\end{remark}

\begin{definition}
Using the embedding in Lemma \ref{L:embeds}, we can define the \emph{relative singularity category} of the pair $(R, A)$ as the triangulated quotient category
\begin{align}\label{E:Def-rel-sing-cat} \Delta_{R}(A)=\frac{\cd^b(\mod-A)}{K^b(\proj-R)}. \end{align}
\end{definition}

\begin{remark}
As a projective $A$-module $eA$ has no self-extensions, therefore $\Delta_{R}(A) \cong \cd^b(\mod-A)/\thick(eA)$ is a relative singularity category in the sense of Chen \cite{Chen11}. Different notions of relative singularity categories were introduced and studied by Positselski \cite{Positselski11} and also by Burke \& Walker \cite{BurkeWalker12}. We thank Greg Stevenson for bringing this unfortunate coincidence to our attention. 
\end{remark}

Let $G'$ be another finitely generated $R$-module, which contains $F'$ as a direct summand. As above, we define $G=R \oplus G'$, $A'=\End_{R}(G)$ and $e'=\id_{R} \in A'$. 

We compare the relative singularity categories of $A$ and $A'$ respectively.

\begin{proposition}\label{P:Embedding-of-relative}
If $A$ has finite global dimension, then there is a fully faithful triangle functor 
\begin{align}
\Delta_{R}(A) \longrightarrow \Delta_{R}(A').
\end{align}
\end{proposition}
\begin{proof}
There is an idempotent $f \in A'$ such that $A \cong fA'f$. This yields a commutative diagram
\begin{align}
\begin{array}{c}
\begin{xy}
\SelectTips{cm}{}
\xymatrix{
&&    \ar[lld]_{\displaystyle -\otimes_{e'A'e'}e'A'f \qquad \, \, \, }     K^b(\proj-e' A' e') \ar[rrd]^{\displaystyle \qquad -\otimes_{e'A'e'}e'A'} \\
K^b(\proj-fA'f) \ar[rrrr]^{\displaystyle - \otimes_{fA'f}fA'}   &&&&     K^b(\proj-A')
}
\end{xy}
\end{array}
\end{align} 
The two `diagonal' functors are the embeddings from Lemma \ref{L:embeds} and the horizontal functor is fully faithful by the same argument. Since $fA'f$ has finite global dimension, this functor yields an embedding $\cd^b(\mod-fA'f) \ra \cd^b(\mod-A')$. Passing to the triangulated quotient categories yields the claim.
\end{proof}

\begin{remark}
We show that the assumption on the global dimension of $A$ is necessary: consider the nodal curve singularity
$A=R=k\llbracket x, y \rrbracket/(xy)$ and its Auslander algebra $A'=\End_{R}(R \oplus k\llbracket x \rrbracket \oplus k\llbracket y \rrbracket)$. In this situation Proposition \ref{P:Embedding-of-relative} would yield an embedding $\underline{\MCM}(R)=\Delta_{R}(R) \rightarrow \Delta_{R}(A').$ But, $\underline{\MCM}(R)$ contains an indecomposable object $X$ with $X \cong X[2s] $ for all $s \in \mathbb{Z}$. Whereas, $\Delta_{R}(A')$ does not contain such objects by the explicit description obtained in \cite[Section 4]{BurbanKalck11}. Contradiction. 
\end{remark}

Without restriction we may assume that $F'$ has no projective direct summands.

\begin{proposition}\label{C:From-gen-to-class}
Let $A$ be a partial resolution. There exists a triangle equivalence
\begin{align}
\frac{\Delta_{R}(A)}{\thick_{\cd^b(\mod-A)}(\mod-A/AeA)} \longrightarrow \ul{\MCM}(R).
\end{align}
\end{proposition}
\begin{proof}
Buchweitz has shown that there exists an equivalence of triangulated categories $\ul{\MCM}(R) \cong \cd_{sg}(R)$ \cite{Buchweitz87}. We have an isomorphism of rings $R \cong eAe$. Hence, the special case $f=e$ of Proposition \ref{P:TwoIdempotents} yields a triangle equivalence
\begin{align} \label{E:OneIdemp}
\frac{\cd^b(\mod-A)/\thick(eA)}{\thick(q(\mod-A/AeA))} \longrightarrow \cd_{sg}(R).
\end{align}
Verdier's \cite[Proposition II.2.3.3]{Verdier} implies that 
\[
q\colon \cd^b(\mod-A) \to \cd^b(\mod-A)/\thick(eA)
\]
induces a (canonical) triangle equivalence
\begin{align}
\thick_{\cd^b(\mod-A)}(\mod-A/AeA) \cong \thick_{\Delta_{R}(A)}(q(\mod-A/AeA)),
\end{align} since there are no non-trivial morphisms from $\thick(eA)$ to $\thick(\mod-A/AeA)$. In the sequel, we identify these two categories via this equivalence. 
\end{proof}

Therefore there is a quotient functor from the relative singularity category $\Delta_{R}(A)$ to the singularity category $\ul{\MCM}(R)$ . 
We want to give an intrinsic description of its kernel $\thick_{\cd^b(\mod-A)}(\mod-A/AeA)$ inside $\Delta_{R}(A)$. We need some preparation. 

\begin{proposition}\label{P:Intrinsic}
In the notations of Propositions \ref{p:recollement-from-projective-general-case} and  \ref{C:From-gen-to-class} assume additionally that $A$ has finite global dimension and $A/AeA$ is finite-dimensional. 

Then there exists a non-positive dg algebra $B$ and a commutative diagram 
\begin{align}
\begin{array}{c}
\begin{xy}
\SelectTips{cm}{}
\xymatrix{ \thick_{\cd^b(\mod-A)}(\mod-A/AeA) \ar@{^{(}->}[r] & \Delta_{R}(A) \ar[d]_\cong^{\displaystyle i^{*}} \ar@{->>}[r]   & \ul{\MCM}(R)   \ar[d]_\cong^{\displaystyle\I} \\
\cd_{fd}(B) \ar[u]^\cong_{\displaystyle i_{*}}  \ar@{^{(}->}[r]  & \per(B) \ar@{->>}[r]  & \per(B)/\cd_{fd}(B) }
\end{xy}
\end{array}
\end{align}
where the horizontal arrows denote (functors induced by) the canonical inclusions and projections respectively. Finally, the triangle functor $\I$ is induced by $i^*$.
\end{proposition}
\begin{proof}
Firstly, $\Delta_{R}(A)$ is idempotent complete: using Schlichting's negative K-Theory for triangulated categories \cite{Schlichting06} this may be deduced from the idempotent completeness of $\ul{\MCM}(R)$ (see \cite[Theorem 3.2.]{BurbanKalck11}). Since $A$ has finite global dimension Corollary \ref{c:restriction-and-induction} implies the existence of a dg $k$-algebra $B$ with $i^*\colon \cd^b(\mod-A)/\thick(eA) \cong \per(B)$. Moreover, since $\dim_{k}(A/AeA)$ is finite $i_{*}\colon \cd_{fd}(B) \cong \thick(\mod-A/AeA)$ by the same corollary. The inclusion 
$\thick_{\cd^b(\mod-A)}(\mod-A/AeA) \hookrightarrow \Delta_{R}(A)$ is induced by the inclusion $\mod-A/AeA \hookrightarrow \mod-A$ (see the proof of Proposition \ref{C:From-gen-to-class}). Since $i_{*}$ and $i^*$ are part of a recollement (Proposition \ref{p:recollement-from-projective-general-case}) we obtain $i^*\circ i_{*}=\id_{\cd(B)}$. 
Hence, the first square commutes. The second square commutes by definition of $\I$. 
\end{proof}
Note that under the assumptions of Proposition~\ref{P:Intrinsic}, we have equalities
\begin{align}\label{E:Dfd}
\cd_{fd,A/AeA}(A)=\thick_{\cd^b(\mod-A)}(\mod-A/AeA)=\cd^b_{A/AeA}(\mod-A).
\end{align}
Moreover, combining this Proposition with Proposition \ref{p:hom-finiteness-of-per} yields the following.
\begin{proposition}\label{p:hom-finiteness-of-delta}
In the setup of Prop. \ref{P:Intrinsic}, the category $\Delta_{R}(A)$ is $\Hom$--finite.
\end{proposition}

\begin{definition} \label{D:Homolog-fin}
For a triangulated category $\ct$ the full triangulated subcategory 

\begin{equation*}
\ct_{rhf}= \left\{ X \in \ct \, \left| \,\,  \Hom_{\ct}(Y, X[i])=0  \text{  for all  } Y \in \ct \text{  and all but finitely many } i \in \mathbb{Z} \right\}\right. \end{equation*}
is called \emph{subcategory of right homologically finite objects}. Dually one defines the subcategory $\ct_{lhf}$Ê of \emph{left homologically finite objects}, see also \cite[Definition 1.6.]{Orlov09}. 
\end{definition}

\begin{example} \label{E:Homolog-fin}
Let $B$ be a dg $k$-algebra such that $\per(B)$ is Hom-finite over $k$, then $\per(B)_{rhf}=\cd_{fd}(B) \cap \per(B)$ since $\Hom(B, X[i]) \cong H^i(X)$ for any dg $B$-module $X$.  In particular, if $\cd_{fd}(B) \subseteq \per(B)$, then $\per(B)_{rhf}=\cd_{fd}(B)$.
\end{example}

\begin{corollary}\label{C:Intrinsic}
In the setup of Proposition \ref{P:Intrinsic}  there is an equality 
\begin{align} 
\thick_{\cd^b(\mod-A)}(\mod-A/AeA)=\Delta_{R}(A)_{rhf},
\end{align}
where we identify $\thick_{\cd^b(\mod-A)}(\mod-A/AeA)$ with its image in $\Delta_{R}(A)$.
\end{corollary}
\begin{proof} One can check that subcategories of right homologically finite objects are identified under triangle equivalences. Therefore, Proposition \ref{P:Intrinsic} shows that  $i^*$ induces an equivalence $\Delta_{R}(A)_{rhf} \cong \per(B)_{rhf}$. By Corollary \ref{C:Hom-finite} $\per(B)$ is Hom-finite and $\cd_{fd}(B) \subseteq \per(B)$. So  Example \ref{E:Homolog-fin} yields $\per(B)_{rhf}=\cd_{fd}(B)$. Using Proposition \ref{P:Intrinsic} again, we see that also 
$i^*(\thick_{\cd^b(\mod-A)}(\mod-A/AeA))=\cd_{fd}(B)=\per(B)_{rhf}$. This finishes the proof.
\end{proof}

\subsection{Intermezzo: Orlov's singularity category for $\Delta_R(A)$}

Orlov defines a singularity category for any triangulated category $\ct$ as follows (\cite[Definition 1.7.]{Orlov09})
\begin{align}
\ct_{sg}:=\ct/\ct_{lhf}.
\end{align}
If $X$ is a separated Noetherian $k$-scheme of finite Krull dimension having enough locally free sheaves, then 
Orlov's \cite[Proposition 1.11.]{Orlov09} shows
\begin{align}
\cd^b(\Coh X)_{sg} \cong \cd_{sg}(X). 
\end{align}
We show that Orlov's singularity category of the relative singularity category $\Delta_R(A)$ is equivalent to the classical singularity category as well

\begin{proposition}\label{R:Orlov}
Let $(A, R)$ be a pair of the following form
\begin{itemize}
\item[(a)] $A$ is the Auslander algebra of an MCM-finite Gorenstein singularity $R$;
\item[(b)] $A$ is the Auslander algebra of a representation-finite selfinjective algebra $R$;
\item[(c)] $A=\End(R \oplus M)$ is an NCCR, where $R$ is an isolated Gorenstein singularity of Krull dimension $d \geq 2$.
\end{itemize}
Then  $\Delta_R(A)_{lhf}=\thick_{\cd^b(\mod-A)}(\mod-A/AeA)=\Delta_R(A)_{rhf}$ inducing an equivalence of triangulated categories
\begin{align}
\Delta_R(A)_{sg} \cong \cd_{sg}(R).
\end{align}
\end{proposition}

\begin{proof}
Our assumptions on $(A, R)$, imply the existence of $n$-almost split sequences and therefore allow us to apply Theorem \ref{t:fractionally-cy-property} later. Namely, for the cases (a), (b) the existence of $1$-almost split sequences (=AR-sequences) on $\MCM(R)$ and $\mod-R$ respectively is well-known. For (c), Iyama's \cite[Theorem 5.2.1]{Iyama07} shows that $R \oplus M \in \MCM(R)$ is maximal $(d-2)$-orthogonal and therefore $\add_R(R \oplus M)$ has $(d-1)$-almost split sequences by \cite[Theorem 3.3.1]{Iyama07a}.

Moreover, under our assumptions, the algebra $A$ has finite global dimension and $A/AeA$ is finite-dimensional. So by Corollary~\ref{C:Hom-finite}, the dg algebra $B$ has finite-dimensional cohomologies in each degree and we can apply Proposition \ref{P:Intrinsic} and Corollary \ref{C:Intrinsic} to reduce the statement to the equality
$\cd_{fd}(B)=\per(B)_{lhf}$.

We prove $\cd_{fd}(B)\subseteq\per(B)_{lhf}$. Recall from Proposition~\ref{p:recollement-from-projective-general-case} and Corollary~\ref{c:restriction-and-induction} that $i_*:\cd(B)\rightarrow \cd(A)$ is fully faithful and the essential image of $\cd_{fd}(B)$ under $i_*$ is $\thick_{\cd^b(\mod-A)}(\mod-A/AeA)$, which is invariant under the Nakayama functor $\nu_A$ on $\cd(A)$ by the proof of Theorem~\ref{t:fractionally-cy-property}. Thus for $X\in\cd_{fd}(B)$, there exists $Z\in\cd_{fd}(B)$ such that $\nu_A(i_*(X))\cong i_*(Z)$. So for any $Y\in\per(B)$ the spaces 
\begin{align*}
\bigoplus_{p\in\mathbb{Z}}D\Hom_{\cd(B)}(X,Y[p])&\cong\bigoplus_{p\in\mathbb{Z}}D\Hom_{\cd(A)}(i_*(X),i_*(Y)[p])\\
&\cong\bigoplus_{p\in\mathbb{Z}}\Hom_{\cd(A)}(i_*(Y),\nu_A(i_*(X))[p])\\
&\cong\bigoplus_{p\in\mathbb{Z}}\Hom_{\cd(A)}(i_*(Y),i_*(Z)[p])\\
&\cong\bigoplus_{p\in\mathbb{Z}}\Hom_{\cd(B)}(Y,Z[p])
\end{align*}
are finite-dimensional, showing that $X\in\per(B)_{lhf}$.

To show the equality $\cd_{fd}(B)=\per(B)_{lhf}$, we claim that the Nakayama functor $\nu_B$ on $\cd(B)$ restricts to the Serre functor on $\cd_{fd}(B)$, which (using the equivalence $i^*$) exists by Theorem \ref{t:fractionally-cy-property} (b). We have the following chain of implications
\begin{align*}
X \in \per(B)_{lhf} &\Rightarrow  \, \bigoplus_{p\in\mathbb{Z}}\Hom_{\per(B)}(X, B[p]) \, \text{is finite-dimensional} \\
&\hspace{-2.9pt}\overset{\eqref{E:AR-formula}}{\Rightarrow} \, \bigoplus_{p\in\mathbb{Z}}H^p(\nu_B(X))=\bigoplus_{p\in\mathbb{Z}}\Hom_{\per(B)}(B, \nu_B(X)[p])\, \text{is finite-dimensional}\\
&\Rightarrow \, \nu_B(X) \in \cd_{fd}(B).
\end{align*}
Since $\cd_{fd}(B) \subseteq \per(B)_{lhf}$ this chain of implications shows that $\nu_B(\cd_{fd}(B)) \subseteq \cd_{fd}(B)$ and therefore by \eqref{E:AR-formula} that $\nu_B$ restricts a right Serre functor on 
$\cd_{fd}(B)$. Since  $\cd_{fd}(B)$ has a Serre functor every right Serre functor is actually a Serre functor by Yoneda's Lemma. This completes the proof of the claim.

Because $B$ has finite-dimensional cohomologies in each degree, $\nu_B$ restricts to a triangle equivalence (see Section~\ref{ss:nakayama-functor})
\[
\nu_B: \per(B)\rightarrow\thick(D(B)),
\]
which, in conjunction with the above chain of implications, shows that $\nu_B$ restricts to a fully faithful functor
\[
\nu_B:\per(B)_{lhf}\rightarrow\cd_{fd}(B).
\]
But, as shown above, the composition
\[
\cd_{fd}(B)\hookrightarrow\per(B)_{lhf}\stackrel{\nu_B}{\rightarrow}\cd_{fd}(B)
\]
is the Serre functor on $\cd_{fd}(B)$, in particular, it is an equivalence. So the inclusion $\cd_{fd}(B)\hookrightarrow\per(B)_{lhf}$ has to be an equality.
\end{proof}

\subsection{Main result}

Now, we are able to state and prove the main result of this article. In particular, it applies to the ADE--singularities, which are listed above. 

\begin{theorem}\label{t:Classical-versus-generalized-singularity-categories}
If $R$ and $R'$ are $\MCM$--finite and $A=\Aus(\MCM(R))$ respectively $A'=\Aus(\MCM(R'))$ denote the Auslander algebras, then the following are equivalent.
\begin{itemize}
\item[(a)] There exists an additive equivalence $\ul{\MCM}(R) \cong \ul{\MCM}(R')$, which respects the action of the respective Auslander--Reiten translations on objects.
\item[(b)] There is an equivalence $\ul{\MCM}(R) \cong \ul{\MCM}(R')$ of triangulated categories.
\item[(c)] There exists a triangle equivalence $\Delta_{R}\!\left(A\right) \cong \Delta_{R'}\!\left(A'\right)$.
\end{itemize}
Moreover, the implication $[(c) \Rightarrow (b)]$ (and hence also $[(c) \Rightarrow (a)]$) holds under much weaker assumptions. Namely, if $A$ and $A'$ are non-commutative resolutions of isolated Gorenstein singularities $R$ and $R'$ respectively.
\end{theorem}
\begin{proof} $[(b) \Rightarrow (a)]$ Clear.

$[(a) \Rightarrow (c)]$ Let $R$ be $\MCM$--finite. It is sufficient to show that the Frobenius category $\MCM(R)$ satisfies the assumptions of Theorem \ref{t:main-thm-2}. Indeed, this implies 
\begin{align}\label{E:rel-sing-cat-as-per}
\Delta_{R}(A) \cong \per(\Lambda_{dg}(\ul{\MCM}(R)), 
\end{align} since $\Delta_{R}(A)$ is idempotent complete by \cite[Theorem 3.2]{BurbanKalck11}. But by construction the dg Auslander algebra $\Lambda_{dg}(\ul{\MCM}(R))$ only depends on the additive structure of $\ul{\MCM}(R)$ and the action of its Auslander--Reiten translation on objects. The claim follows.

The assumptions, which we have to verify are: existence of almost split sequences in $\MCM(R)$; $\Hom$-finiteness and idempotent completeness of the stable category $\ul{\MCM}(R)$. The last property follows from idempotent completeness of $\MCM(R)$ and the existence of lifts of idempotents from $\ul{\MCM}(R)$ to $\MCM(R)$, which holds since $R$ is complete. The first two assertions were shown by M. Auslander. Precisely, in our situation $R$-lattices (\confer \cite[Appendix]{Auslander84}) are Cohen--Macaulay, hence $\MCM$--finiteness and \cite[Corollary A.2]{Auslander84} imply that $R$ is an isolated singularity. Then the notions of Cohen--Macaulay $R$-modules and $R$-lattices coincide. Now, the main theorem in \emph{op.~cit.} completes the proof.

$[(c) \Rightarrow (b)]$ We claim that this is a consequence of Proposition \ref{C:From-gen-to-class} and Corollary \ref{C:Intrinsic}. Indeed, by Proposition \ref{C:From-gen-to-class} the stable category $\ul{\MCM}(R)$ is a quotient of $\Delta_{R}(A)$ and by Corollary \ref{C:Intrinsic} the kernel $\cd^b_{A/AeA}(\mod-A) \subseteq \Delta_{R}(A)$ of the quotient functor has an intrinsic characterization. Hence, the triangle equivalence in (c) induces an equivalence between the respective quotient categories as in (b).

We verify the (stronger) assumptions of Corollary \ref{C:Intrinsic}. $\Hom$-finiteness of $\ul{\MCM}(R)$ follows as in the proof of $[(a) \Rightarrow (c)]$ and holds more generally for any (complete) isolated singularity $R$. In particular, the algebra $A/AeA$ is finite-dimensional. The Auslander algebra of $\MCM(R)$ has finite global dimension by \cite[Theorem A.1]{Auslander84}.
\end{proof}

\begin{example}\label{E:Conifold}
Let $R=\C\llbracket u, v  \rrbracket/(uv)$ and $R''=\C\llbracket u, v, w, x  \rrbracket/(uv+ wx)$ be the one and three dimensional $A_{1}$--singularities, respectively. The latter is also known as the ``conifold''. The Auslander--Reiten quivers $\ca(R)$ and $\ca(R'')$ of $\MCM(R)$ respectively $\MCM(R'')$, are known, \confer \cite{Schreyer87} (in particular, \cite[Remark 6.3]{Schreyer87} in dimensions $\geq 3$):

 \begin{equation*}\begin{tikzpicture}[description/.style={fill=white,inner sep=2pt}]

    \matrix (n) [matrix of math nodes, row sep=2em,
                 column sep=1.5em, text height=1.5ex, text depth=0.25ex,
                 inner sep=0pt, nodes={inner xsep=0.3333em, inner
ysep=0.3333em}] at (0, 0)
    {+ && *  && - &&&&&             +  && \star &&- \\};
  
 \node at ($(n-1-1.west) + (-8mm, 0mm)$) {$\ca(R)=$};
 \node at ($(n-1-1.west) + (64mm, 0mm)$) {$\ca(R'')=$};
 
     \draw[->] ($(n-1-1.east) + (0mm,1mm)$) .. controls +(4.5mm,2mm) and
+(-4.5mm,+2mm) .. ($(n-1-3.west) + (0mm,1mm)$);

    \draw[->] ($(n-1-3.east) + (0mm,1mm)$) .. controls +(4.5mm,2mm) and
+(-4.5mm,+2mm) .. ($(n-1-5.west) + (0mm,1mm)$);

    \draw[->] ($(n-1-3.west) + (0mm,-1mm)$) .. controls +(-4.5mm,-2mm)
and +(+4.5mm,-2mm) .. ($(n-1-1.east) + (0mm,-1mm)$);
 
  \draw[->] ($(n-1-5.west) + (0mm,-1mm)$) .. controls +(-4.5mm,-2mm)
and +(+4.5mm,-2mm) .. ($(n-1-3.east) + (0mm,-1mm)$);

 \path[dash pattern = on 0.5mm off 0.3mm, ->] ($(n-1-1.north) + (0.5 mm,-0.5mm)$) edge [bend left=50] ($(n-1-5.north) + (-.5mm,-.5mm)$);

 \path[dash pattern = on 0.5mm off 0.3mm, <-] ($(n-1-1.south) + (.5 mm,.5mm)$) edge [bend right=50]  ($(n-1-5.south) + (-.5mm,.5mm)$);


    \draw[->] ($(n-1-10.east) + (0mm,1mm)$) .. controls +(4.5mm,2mm) and
+(-4.5mm,+2mm) .. ($(n-1-12.west) + (0mm,1mm)$);

    \draw[->] ($(n-1-12.east) + (0mm,1mm)$) .. controls +(4.5mm,2mm) and
+(-4.5mm,+2mm) .. ($(n-1-14.west) + (0mm,1mm)$);

    \draw[->] ($(n-1-12.west) + (0mm,-1mm)$) .. controls +(-4.5mm,-2mm)
and +(+4.5mm,-2mm) .. ($(n-1-10.east) + (0mm,-1mm)$);

    \draw[->] ($(n-1-14.west) + (1mm,-2mm)$) .. controls +(-3.7mm,-3mm)
and +(+3.7mm,-3mm) .. ($(n-1-12.east) + (-1mm,-2mm)$);

    \draw[->] ($(n-1-10.east) + (-1mm,2mm)$) .. controls +(3.7mm,3mm) and
+(-3.7mm,+3mm) .. ($(n-1-12.west) + (1mm,2mm)$);

    \draw[->] ($(n-1-12.east) + (-1mm,2mm)$) .. controls +(3.7mm,3mm) and
+(-3.7mm,+3mm) .. ($(n-1-14.west) + (1mm,2mm)$);

    \draw[->] ($(n-1-12.west) + (1mm,-2mm)$) .. controls +(-3.7mm,-3mm)
and +(+3.7mm,-3mm) .. ($(n-1-10.east) + (-1mm,-2mm)$);

 \draw[->] ($(n-1-14.west) + (0mm,-1mm)$) .. controls +(-4.5mm,-2mm)
and +(+4.5mm,-2mm) .. ($(n-1-12.east) + (0mm,-1mm)$);

 \path[dash pattern = on 0.5mm off 0.3mm, ->] ($(n-1-10.north) + (0.5 mm,-0.5mm)$) edge [bend left=50] ($(n-1-14.north) + (-.5mm,-.5mm)$);

 \path[dash pattern = on 0.5mm off 0.3mm, <-] ($(n-1-10.south) + (.5 mm,.5mm)$) edge [bend right=50]  ($(n-1-14.south) + (-.5mm,.5mm)$);

\end{tikzpicture}\end{equation*}
Let $A$ and $A''$ be the respective Auslander algebras of $\MCM(R)$ and $\MCM(R'')$. They are given as quivers as quivers with relations, where the quivers are just the ``solid'' subquivers of $\ca(R)$ and $\ca(R'')$, respectively. Now, Kn\"orrer's Periodicity Theorem \ref{t:knoerrer} and Theorem \ref{t:Classical-versus-generalized-singularity-categories} above show that there is an equivalence of triangulated categories
\begin{align}
\frac{\cd^b(\mod-A)}{K^b(\add P_{*})} \longrightarrow \frac{\cd^b(\mod-A'')}{K^b(\add P_\star)}, 
\end{align}
where $P_{*}$ is the indecomposable projective $A$-module corresponding to the vertex $*$ and similarly $P_{\star} \in \proj-A''$ corresponds to $\star$. 

Note, that the relative singularity category $\Delta_{R}(A)=\cd^b(\mod-A)/K^b(\add P_{*})$ from above has an explicit description, see \cite[Section 4]{BurbanKalck11}.  

\end{example}

\begin{remark}
For finite-dimensional selfinjective $k$-algebras of finite representation type one can prove (the analogue of) implication [(b) $\Rightarrow$ (c)] in Theorem \ref{t:Classical-versus-generalized-singularity-categories} above without relying on dg--techniques. Indeed, Asashiba \cite[Corollary 2.2.]{Asashiba99} has shown that in this context stable equivalence implies derived equivalence. Now, Rickard's \cite[Corollary 5.5.]{Rickard91} implies that the respective Auslander algebras are derived equivalent (a result, which was recently obtained by W. Hu and C.C. Xi in a much more general framework \cite[Corollary 3.13]{HuXi09}\footnote{The first author would like to thank Sefi Ladkani for pointing out this reference.}). One checks that this equivalence induces a triangle equivalence between the respective relative singularity categories. This result is stronger than the analogue of Theorem \ref{t:Classical-versus-generalized-singularity-categories} (c).
\end{remark}
\subsection{Global relative singularity categories}\label{ss:Global}
Let $X$ be a quasi-projective scheme and $\cf$ a coherent sheaf, which is locally free on  $X \setminus \mathsf{Sing}(X)$. We assume that $\ca={\mathcal End}_{X}(\co_{X} \oplus \cf)$ has finite global dimension. Hence, the ringed space $\mathbb{X}=(X, \ca)$ is a non-commutative resolution of $X$ and  $\cd^b(\Coh(\mathbb{X}))$ is a categorical resolution in the spirit of works of Van den Bergh \cite{VandenBergh04}, Kuznetsov \cite{Kuznetsov08} and Lunts \cite{Lunts10b}. There is a triangle embedding $\mathsf{Perf}(X) \ra \cd^b(\Coh(\mathbb{X}))$. Thus, we can define the \emph{relative singularity category} as the idempotent completion \cite{BalmerSchlichting01} of the  corresponding triangulated quotient category: 
$
\Delta_{X}(\mathbb{X})= \left(\cd^b(\Coh(\mathbb{X}))/\mathsf{Perf}(X)\right)^\omega.
$
If $X$ has \emph{isolated} singularities, then the study of $\Delta_{X}(\mathbb{X})$ reduces to  the ``local'' relative singularity categories defined above. Precisely, there exists an triangle equivalence \cite[Cor. 2.11.]{BurbanKalck11}
\begin{align}
\Delta_{X}(\mathbb{X}) \cong \bigoplus_{x \in \mathsf{Sing}(X)} \Delta_{\widehat{\co}_{x}}(\widehat{\ca}_{x}).
\end{align}
If $X$ is a curve with nodal singularities, then this yields a complete and explicit description of the category $\Delta_{X}(\mathbb{X})$, where $\ca$ is the \emph{Auslander sheaf} of $X$ \cite{BurbanKalck11}.

 \section{Related work} \label{s:Concluding}
 \subsection{Relationship to Bridgeland's moduli space of stability conditions}\label{ss:Bridgeland} Let $X=\mathsf{Spec}(R_{Q})$ be a Kleinian singularity with \emph{minimal} resolution $f\colon Y \ra X$ and exceptional divisor $E=f^{-1}(0)$. Then $E$ is a tree of rational $(-2)$--curves, whose dual graph $Q$ is of ADE--type. Let us consider the following triangulated category
\begin{align}
\cd = \ker\left(\mathbb{R}f_{*}\colon \cd^b(\Coh(Y)) \longrightarrow \cd^b(\Coh(X))\right).
\end{align}
 Bridgeland determined a connected component $\mathsf{Stab}^\dagger(\cd)$ of the stability manifold of $\cd$ 
 \cite{Bridgeland09}. More precisely, he proves that $\mathsf{Stab}^\dagger(\cd)$ is a covering space of $\mathfrak{h}^{\rm reg}/W$, where $\mathfrak{h}^{ \rm reg}\subseteq \mathfrak{h}$ is the complement of the root hyperplanes in a fixed Cartan subalgebra $\mathfrak{h}$ of the complex semi--simple Lie algebra $\mathfrak{g}$ of type $Q$ and $W$ is the associated Weyl group. It turns out, that $\mathsf{Stab}^\dagger(\cd)$ is even a \emph{universal} covering of  $\mathfrak{h}^{\rm reg}/W$. This follows \cite{Bridgeland09} from a faithfulness result for the braid group actions generated by spherical twists  (see \cite{SeidelThomas01} for type $A$ and \cite{BravThomas11} for general Dynkin types).
 
The category $\cd$ admits a different description. Namely, as category of dg modules with finite-dimensional total cohomology $\cd_{fd}(B)$, where $B=B_{Q}$ is the dg-Auslander algebra $\Lambda_{dg}(\ul{\MCM}(R))$ of $R=\widehat{R}_{Q}$. Let $A=\Aus(\MCM(R))$ be the Auslander algebra of $\MCM(R)$ and denote by $e$ the identity endomorphism of $R$ considered as an idempotent in $A$. Then the derived McKay--Correspondence \cite{KapranovVasserot00, BridgelandKingReid01} induces a commutative diagram of triangulated categories and functors, \confer~\cite[Section 1.1]{Bridgeland09}.
\begin{align}
\begin{array}{c}
\begin{xy}
\SelectTips{cm}{}
\xymatrix{
\cd \ar@{=}[r] \ar[d]^{\cong} &\ker\big(\mathbb{R}f_{*}\colon \cd^b(\mathsf{Coh}(Y)) \ra \cd^b(\Coh(X))\big) \ar@{^{(}->}[r] \ar[d]_{\cong} & \cd^b_{E}(\Coh(Y)) \ar[d]^{\cong}\\
\cd_{fd}(B) \ar[r]^(0.4)\cong & \cd^b_{A/AeA}(\mod-A) \ar@{^{(}->}[r] & \cd^b_{fd}(\mod-A).
}
\end{xy}
\end{array}
\end{align}
Here $\cd^b_{fd}(\mod A)$ denotes the full subcategory of $\cd^b(\mod A)$ consisting of complexes whose total cohomology is finite-dimensional. Thanks to \cite[Lemma 2.3]{BurbanKalck11}, the canonical functor $\cd^b(\fdmod A)\rightarrow\cd^b_{fd}(\mod A)$ is an equivalence. For the equivalence $\cd_{fd}(B) \cong \cd^b_{A/AeA}(\mod-A)$, we refer to Proposition \ref{P:Intrinsic} and (\ref{E:Dfd}). Moreover, this category is triangle equivalent to the kernel of the quotient functor $\Delta_{R}(A) \ra \cd_{sg}(R)$, see Proposition \ref{C:From-gen-to-class}.

\begin{remark} 
It would be interesting to study Bridgeland's space of stability conditions for the categories $\cd_{fd}(B)$ in the case of  odd dimensional ADE--singularities $R$ as well! Note that the canonical $t$-structure on $\cd(B)$ restricts to a $t$-structure on $\cd_{fd}(B)$ by Proposition \ref{p:standard-t-str}. Its heart is the finite length category of finite-dimensional modules over the stable Auslander algebra of $\MCM(R)$.
\end{remark}

\subsection{Links to generalized cluster categories}\label{ss:links}
Let $k$ be an algebraically closed field of characteristic $0$. Let $Q$ be a quiver of ADE--type. As above, we consider the dg Auslander algebra $B_{Q}=\Lambda_{dg}\left(\ul{\MCM}(\widehat{R}_{Q})\right)$ of the corresponding ADE--singularity $\widehat{R}_{Q}$ of even Krull dimension. There exists an isomorphism of dg algebras
\begin{align}
B_{Q} \cong \Pi(Q, 2, 0),
\end{align} 
where $\Pi(Q, d, W)$ denotes the deformed dg preprojective algebra, which was associated to a finite (graded) quiver $Q$, a positive integer $d$ and a potential $W$ of degree $-d+3$ by Ginzburg \cite{Ginzburg06} (see also \cite{VanDenBergh10}). 

$B_{Q}$ is a bimodule $2$-Calabi--Yau algebra in the sense of \cite{Ginzburg06}. Hence, the triangle equivalence (\ref{E:Cluster-Equiv}) yields the well-known result that $\ul{\MCM}(\widehat{R}_{Q})$ is the $1$--cluster category of $kQ$  (see e.g.~Reiten \cite{Reiten87}).
More generally, Van den Bergh's \cite[Theorem 10.2.2]{VanDenBergh10} shows that $\Pi(Q, d, W)$ is bimodule $d$-Calabi--Yau\footnote{More precisely, Van den Bergh proves that $\Pi(Q, d, W)$ is \emph{exact} Calabi--Yau, which implies the bimodule Calabi--Yau property.}. 

Now, if $H^0\!\left(\Pi(Q, d, W)\right)$ is finite-dimensional, then (by definition) the quotient 
\begin{align}
\cc_{(Q, d, W)}=\frac{\per\big(\Pi(Q, d, W)\big)}{\cd_{fd}\big(\Pi(Q, d, W)\big)}
\end{align}
is a generalized $(d-1)$-cluster category. In particular, $\cc_{(Q,d,W)}$ is $(d-1)$-Calabi--Yau and the image of $\Pi(Q, d, W)$ defines a $(d-1)$-cluster tilting object \cite{Amiot09,Guolingyan11a}.

The following Morita-type question attracted a lot of interest recently.
\begin{question}\label{Q:Morita}
Let $\cc$ be a $k$-linear $\Hom$-finite $d$-Calabi--Yau algebraic triangulated category with $d$--cluster--tilting object. Is there a triple $(Q, d, W)$ as above such that
$\cc$ is triangle equivalent to the corresponding cluster category  $\cc_{(Q, d, W)}$ \!?
\end{question}

In a recent series of papers Amiot~\emph{et~al.} answer this question to the affirmative 
in some interesting special cases \cite{Amiot09,AmiotReitenTodorov11,AmiotIyamaReitenTodorov10,AmiotIyamaReiten11}. 
In \cite{Amiot10} Amiot gives a nice overview.

Let us outline another promising approach \cite{KalckYang12b} to tackle Question \ref{Q:Morita}: 
a combination of Keller \& Vossieck's \cite[Exemple 2.3]{KellerVossieck87} with the theory developed in this article shows that 
for many interesting algebraic triangulated categories $\ct$ there exists a non-positive dg algebra $B$ 
and a triangle equivalence generalizing (\ref{E:Cluster-Equiv}) above
\begin{align}\label{E:Cluster-type}
\ct \cong \frac{\per(B)}{\cd_{fd}(B)}.
\end{align}
In particular, this holds for stable categories of maximal Cohen--Macaulay modules over certain Iwanaga--Gorenstein rings 
and the Calabi--Yau categories arising from subcategories of nilpotent representations over preprojective algebras 
(\confer \cite{BuanIyamaReitenScott09,GeissLeclercSchroeer06a}). Palu~\cite{Palu09} also obtained such an equivalence (in a 
slightly different form)
in his study of Grothendieck groups of Calabi--Yau categories with cluster-tilting objects.

Now, if $\ct$ is $d$-Calabi--Yau category as in Question \ref{Q:Morita}, then $\cd_{fd}(B)$ is a $(d+1)$-Calabi--Yau category by Keller \& Reiten's \cite[Theorem 5.4]{KellerReiten07}. 
In conjunction with Van den Bergh's \cite[Theorem 10.2.2]{VanDenBergh10}, 
we see that Question \ref{Q:Morita} has an affirmative answer, if the following statement holds (we use the terminology from
\cite{VanDenBergh10}). 

\medskip
\noindent\emph{If $A$ is a pseudo-compact dg algebra such that $\cd_{fd}(A)$ is a $d$-Calabi--Yau triangulated category generated by
a finite number of simple dg $A$-modules, then $A$ is an exact $d$-Calabi--Yau dg algebra.}
\medskip

This statement has a conjectural status in general. However, for some interesting $d$-Calabi--Yau categories $\ct$ 
(with $d$-cluster-tilting object) one can show that $B$ is strongly $d$-Calabi--Yau without relying on the statement above. 
For example, 
this was done by Thanhoffer de V\"olcsey \& Van den Bergh for $\ct=\ul{\MCM}(R)$, 
where $R$ is a complete Gorenstein quotient singularity of Krull dimension three \cite{ThanhofferdeVolcseyMichelVandenBergh10}. 
They also prove (\ref{E:Cluster-type}) in a more restricted setup.

\appendix

\section{DG-Auslander algebras for ADE--singularities}\label{Appendix} 
In this appendix, we work over the field of complex numbers $\C$.

\noindent
 The stable Auslander--Reiten quivers for the curve and surface singularities of Dynkin type ADE are known, see \cite{DieterichWiedemann86} and  \cite{Auslander86} respectively. 
Hence, the stable Auslander--Reiten quiver for any ADE--singularity $R$ is known by Kn\"orrer's periodicity (Theorem \ref{t:knoerrer}). The equivalence (\ref{E:rel-sing-cat-as-per}) in the proof of Theorem \ref{t:Classical-versus-generalized-singularity-categories} describes the triangulated category $\Delta_{R}(\Aus(R))$ as the perfect category for the dg-Auslander algebra associated to $\ul{\MCM}(R)$. We list the graded quivers\footnote{M.K. thanks Hanno Becker for his help with the TikZ--package.} of these dg-algebras for the ADE--singularities in Subsections \ref{ss:OddTypeA} - \ref{ss:even}. For surfaces, this also follows from \cite{ThanhofferdeVolcseyMichelVandenBergh10, 
AmiotIyamaReiten11}.
\begin{remark}
For ADE--singularities $R$,  it is well-known that the stable categories $\ul{\MCM}(R)$ are \emph{standard}, i.e.~the mesh relations form a set of minimal relations for the Auslander algebra $\Aus(\ul{\MCM}(R))$ of $\ul{\MCM}(R)$ (\confer \cite{Amiot07a, Riedtmann80}, respectively \cite{Iyama07a}). Hence, the graded quivers completely determine the dg Auslander algebras in this case.
\end{remark}

The conventions are as follows. Solid arrows $\longrightarrow$ are in degree $0$, whereas broken arrows
$
\begin{xy}
\SelectTips{cm}{}
\xymatrix{ \ar@{-->}[r]&}
\end{xy}
$
are in degree $-1$ and correspond to the action of the Auslander--Reiten translation. The differential $d$ is uniquely determined by sending each broken arrow $\rho$ to the mesh relation starting in $s(\rho)$. If there are no irreducible maps (i.e.~ solid arrows) starting in the vertex $s(\rho)$, then we set $d(\rho)=0$ (\confer the case of type $(A_{1})$ in odd dimension in Subsection \ref{ss:OddTypeA}). Let us illustrate this by means of two examples: in type $(A_{2m})$ in odd Krull dimension (see Subsection \ref{ss:OddTypeA}) we have
\begin{align}Êd(\rho_{2})=\alpha_{1}\alpha_{1}^*+\alpha_{2}^*\alpha_{2}, \end{align}
whereas in odd dimensional type $(E_{8})$  (see Subsection \ref{ss:OddTypeE8}) 
\begin{align} d(\rho_{10})=\alpha_{8}\alpha_{8}^*+ \alpha_{16}\alpha_{16}^*+\alpha_{9}^*\alpha_{10}. \end{align}

\subsection{DG-Auslander algebras for Type $A$--singularities in odd dimension}
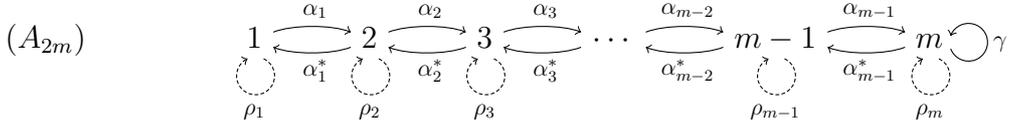
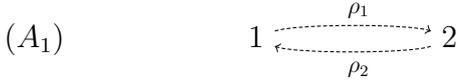
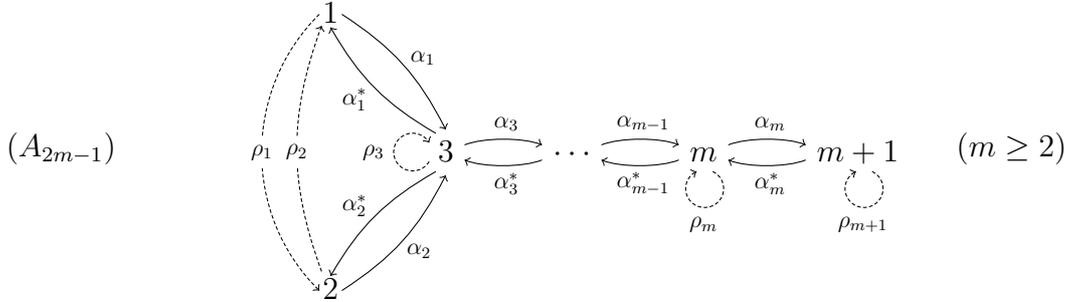
\label{ss:OddTypeA}
 \begin{equation*}
\end{equation*}

\subsection{DG--Auslander algebras of odd dim. $\left(D_{2m+1}\right)$-singularities, ${m\geq2}$}

 \begin{equation*}%
\end{equation*}

\subsection{DG--Auslander algebras of odd dimensional $(D_{2m})$-singularities, $m \hspace{-2pt}\geq\hspace{-2pt} 2$}

 \begin{equation*}%
\end{equation*}

\subsection{The DG--Auslander algebra of odd dimensional $\left(E_{6}\right)$-singularities.}
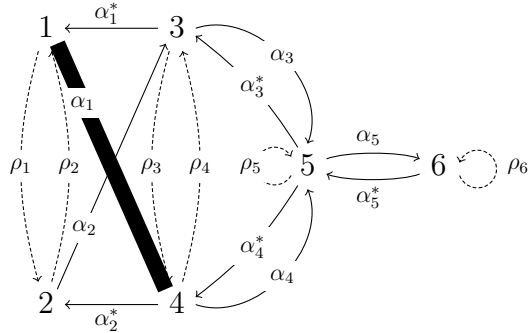
\begin{equation*}%
\end{equation*}

\subsection{The DG--Auslander algebra of odd dimensional $\left(E_{7}\right)$-singularities.}

 \begin{equation*}%
\end{equation*}

\subsection{The DG--Auslander algebra of odd dimensional $\left(E_{8}\right)$-singularities.}
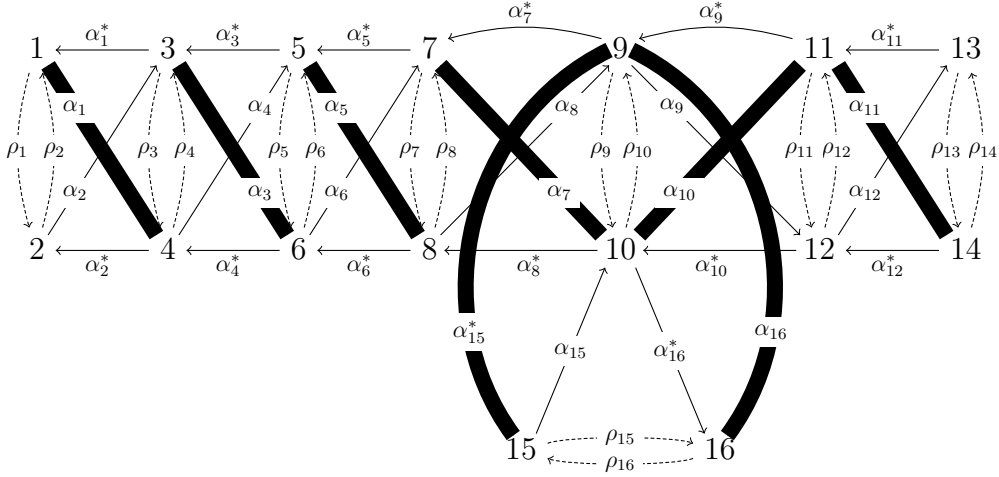
\label{ss:OddTypeE8}

\begin{equation*}%
\end{equation*}


%

%

\subsection{DG-Auslander algebras of even dimensional ADE--singularities }\label{ss:even}

\begin{equation*}%
\end{equation*}


\def\cprime{$'$}
\providecommand{\bysame}{\leavevmode\hbox to3em{\hrulefill}\thinspace}
\providecommand{\MR}{\relax\ifhmode\unskip\space\fi MR }
\providecommand{\MRhref}[2]{%
  \href{http://www.ams.org/mathscinet-getitem?mr=#1}{#2}
}
\providecommand{\href}[2]{#2}

\end{document}